\documentclass[onefignum,onetabnum]{siamart220329}

\usepackage{aliascnt}

\usepackage{graphicx}
\usepackage{color}
\usepackage{latexsym}

\usepackage{booktabs}
\usepackage{todonotes}
\usepackage{amsfonts}
\usepackage{algpseudocode}
\usepackage{algorithm}
\usepackage{caption}
\usepackage{multirow}
\usepackage{float}
\usepackage{hyperref}
\usepackage[normalem]{ulem}
\usepackage{microtype}  
\usepackage{bm}

\usepackage{enumitem}          

\allowdisplaybreaks[3]          
\setlist{noitemsep,topsep=2pt,parsep=1pt,partopsep=0pt}
\definecolor{darkgreen}{rgb}{0,0.5,0}

\headers{Derivative-Free Bilevel Optimization.}{E. Cesaroni, G. Liuzzi, S. Lucidi}

\newtheorem{ass}{Assumption}
\newtheorem{remark}{Remark}
\newtheorem{prop}{Proposition}

\newtheorem{Teorema}{Theorem}
\renewcommand{\Re}{\mathbb{R}}

\title{Derivative-Free Bilevel Optimization with Inexact Lower-Level Solutions\thanks{Submitted to the editors \today.}}

\author{
Edoardo Cesaroni\thanks{Dipartimento di Ingegneria Informatica, Automatica e Gestionale ``A. Ruberti'', ``Sapienza'' Universit\`a di Roma, Via Ariosto 25, 00185 Rome (Italy) (\email{cesaroni@diag.uniroma1.it}, \email{liuzzi@diag.uniroma1.it}, \email{lucidi@diag.uniroma1.it}).} \and
Giampaolo Liuzzi\footnotemark[2] \thanks{Istituto di Analisi dei Sistemi ed Informatica ``A. Ruberti'', Consiglio Nazionale delle Ricerche, via dei Taurini 19, 00185 Rome (Italy)} \and   
Stefano Lucidi\footnotemark[2]
}

\begin{document}
\maketitle

\begin{abstract}
In this work, we propose derivative-free framework for bilevel optimization. We consider both the upper and lower-level problems with bound constraints on the variables, as well as general nonlinear constraints, assuming that first-order information (in the upper-level) is not available or it is impractical to obtain. The lower-level problem is solved with an accuracy that is progressively refined throughout the optimization process.
We first analyze the case in which the upper-level problem is subject only to bound constraints, establishing convergence to Clarke-Jahn stationary points when the refinement process is allowed to reach its maximum precision. When a limitation is imposed on this refinement process, we prove convergence to approximate stationary points using an extended notion of Goldstein stationarity.
Finally, we extend the proposed framework to handle more complex constraints via an exact penalty function approach, proving convergence to stationary points under suitable assumptions. A comprehensive numerical study on 160 problems from the BOLIB collection shows that the adaptive accuracy strategy consistently yields better results than fixed-precision solves, with its benefits becoming more pronounced as the required lower-level accuracy becomes more stringent.
\end{abstract}

\begin{keywords}
Bilevel Optimization, Derivative-Free Methods, Exact Penalty
\end{keywords}

\begin{MSCcodes}
90C26, 90C30, 90C56
\end{MSCcodes}

\section{Introduction}\label{intro}

In this paper we consider bilevel optimization problems of the form:
\begin{equation}\label{prob:upperlevel}
\begin{aligned}
\text{``$\min_{x}$''} \quad & F(x,y) \\[4pt]
\text{s.t.} \quad 
& H(x,y) \le 0, \\[2pt]
& x \in X := \{\, l_i \le x_i \le u_i,\ i=1,\ldots,n_x \,\}, \\[2pt]
& y \in \mathcal{R}(x),
\end{aligned}
\end{equation}
where
\begin{equation}\label{prob:lowerlevel}
\begin{array}{rl}
	{\cal R}(x) = \displaystyle\arg\min_{z} & f(x,z) \\[6pt]
	\text{s.t.} & z \in {\cal Z} \subseteq \mathbb{R}^{n_y}.
\end{array}
\end{equation}
with $F:\Re^{n_x + n_y}\to\Re$, $H:\Re^{n_x + n_y}\to \Re^m$ (upper-level objective and constraint functions) and $f:\Re^{n_x + n_y}\to\Re$ (lower-level objective) and \( l_i, u_i \in \mathbb{R},
\ i = 1, \ldots, n_x \).
\par\smallskip
The set ${\cal R}(x)$, called the \textit{rational reaction set}, plays a central role in the analysis and solution of bilevel problems. In Problem \eqref{prob:upperlevel}, commonly referred to as the upper-level problem, the optimization is performed only with respect to $x$, the variable controlled by the upper-level player (the leader). When the problem is well-posed, the rational reaction set reduces to a singleton, i.e., ${\cal R}(x) = \{y(x)\}$. The quotation marks around ``$\min$'' emphasize the inherent ambiguity arising when the lower-level problem admits multiple optimal solutions: in such ill-posed cases, the leader cannot uniquely predict the follower's response.
This formulation captures a hierarchical decision-making process between two players. The leader announces a decision $x$, and the follower reacts by selecting a value $y$ that minimizes their own objective $f(x,\cdot)$. For a comprehensive treatment of bilevel formulations, we refer the reader to \cite{Zemkoho2016} and the references therein.
\par\smallskip

Over the past years, many methods have been proposed in the literature for the solution of bilevel problems. Solution methods are both derivative-based (i.e., they take advantage of the information provided by the derivatives of the objective function) and derivative-free (i.e., they do not make use of first order derivatives of the objective function). In the derivative-free context, we are particularly interested in the case where the functions defining the upper-level problem are of the black-box-type. 

Genetic approaches include \cite{hejazi2002linear} for linear bilevel programming via Karush-Kuhn-Tucker (KKT) reformulation, and \cite{wang2005evolutionary} with tailored crossover and constraint-handling operators.
In \cite{colson2004trust}, a trust-region method is proposed for bilevel problems with unknown derivatives and upper-level constraints that do not involve lower-level variables. The lower-level objective is approximated quadratically, while the constraints and upper-level objective are approximated linearly.
Another trust-region algorithm for unconstrained bilevel problems is proposed in \cite{conn2012bilevel}, where the trust-region procedure is applied to both the upper- and lower-level problems.

In \cite{mersha2011direct}, a direct-search method for bilevel problems without upper-level constraints and with a convex lower-level problem is proposed. The method relies on a single-level reformulation based on the optimal lower-level reaction, and its convergence is analyzed in a derivative-free setting.
In \cite{diouane2023inexact}, a derivative-free bilevel framework with a fixed-accuracy approximation of the lower-level solution is proposed, for which the authors establish convergence to \((\delta,\epsilon)\)-Goldstein stationary points \cite{Lin2022,diouane2023inexact}.
\par
The method proposed in this work takes inspiration from the works in \cite{mersha2011direct,diouane2023inexact}. In particular, we start with the same type of single level reformulation in order to design our derivative-free method. However, some significant differences and novelties can be highlighted.
First, we develop derivative-free frameworks for bilevel optimization that handle both bound-constrained and generally constrained upper-level problems. 
Second, a key innovation of our approach is the treatment of inexact lower-level solutions: rather than requiring exact solutions at each upper-level iteration, we allow the lower-level problem to be solved with a tolerance $\zeta$ that can be progressively adapted as the optimization proceeds. This adaptive accuracy strategy significantly reduces computational cost in early iterations while ensuring convergence.
\par\smallskip
The paper has two main parts. The first addresses bilevel problems with only bound constraints at the upper level, introducing our derivative-free method and its initial convergence results. The second extends the framework to general nonlinear upper-level inequality constraints. More specifically, the paper is structured as follows.\\
Section~\ref{sec:boundconstr} presents a linesearch-based derivative-free approach (DFN-LLA) for bound-constrained bilevel problems with lower-level adaptive accuracy and analyzes its convergence under different lower-level accuracy settings. Section~\ref{sec:genconstr} incorporates general nonlinear constraints via an exact penalty approach and extends the convergence analysis. Numerical results are reported in Section~\ref{sec:numerical}, while Section~\ref{sec:conclusions} summarizes the main findings and outlines possible directions for future research.

\subsection{Contributions}
The principal contributions of this work are summarized as follows:
\begin{itemize}
\item We introduce a linesearch-based, derivative-free optimization algorithm for bilevel problems and establish its global convergence to Clarke–Jahn stationary points in the regime where the tolerance associated with the lower-level problem is driven to zero.
\item We investigate the setting in which the lower-level solution accuracy is bound\-ed away from zero and prove convergence to approximate stationary points. These points are characterized via a Goldstein-type stationarity condition.
\item We generalize the proposed frameworks to accommodate general nonlinear constraints at the upper-level by means of an exact penalty function approach, and we prove convergence to Clarke–KKT stationary points (and their approximate counterparts) under appropriate regularity assumptions.
\item We conduct an extensive numerical study on a comprehensive bilevel test set, showing that dynamically adapting the lower-level accuracy yields better performance compared to fixed-precision solves, particularly in regimes requiring stringent lower-level tolerances.
\item We extend the performance and data profile methodology to the bilevel optimization context, enabling a systematic benchmarking of bilevel solvers.
\end{itemize}

\subsection{Assumptions}

We now state the assumptions supporting our analysis, concerning selection mappings for lower-level solutions and regularity of the objectives.
Since ${\cal R}(x)$ may not be a singleton for given values of the upper-level variables $x_i$, $i=1,\dots,n_x$ and in practice only approximate lower-level solutions with tolerance $\zeta > 0$ are available, we assume the existence of two mappings: one selecting an exact solution $y(x) \in {\cal R}(x)$, and one providing an approximate solution $\tilde{y}(x, \zeta)$ within distance $\zeta$ from $y(x)$.

\begin{ass}\label{ass:mapping}
For any $x\in X$ and $\zeta > 0$, there exist two mappings   $y:\Re^{n_x}\to\Re^{n_y}$ and $\tilde y:\Re^{n_x}\times\Re\to\Re^{n_y}$ such that:
\begin{enumerate}
    \item[(a)] $y(x)\in{\cal R}(x)$;
    \item[(b)] $\| \tilde y(x,\zeta)-y(x)\|\leq \zeta$.
\end{enumerate}
\end{ass}

Then, for any $x\in X$ and $\zeta>0$, we denote 

\begin{equation}\label{def:functions}
   \begin{split}
    \bar F(x) & = F(x,y(x)),\\
    \bar H_i(x) & = H_i(x,y(x)),\ i=1,\dots,m\\
    P(x,\zeta) & = F(x,\tilde y(x,\zeta)).
    \end{split} 
\end{equation}

\begin{remark}
Concerning Assumption \ref{ass:mapping}, we note that point (a) is surely satisfied when ${\cal R}(x)$ is a singleton. However, our assumption allows for more general situations in which ${\cal R}(x)$ is not a singleton but it is possible to uniquely identify an element $y(x)\in {\cal R}(x)$ for any given $x$. For instance, we recall two main approaches (see e.g., \cite{Zemkoho2016}):
\begin{itemize}
\item[] Optimistic: in this situation, the $y\in {\cal R}(x)$ is selected which is the best one from the upper-level point of view.
\item[] Pessimistic: this approach tries to bound the possible damage that the worst choice of $y\in {\cal R}(x)$ can cause to the upper-level objective function.    
\end{itemize}
As for point (b), it can be satisfied by assuming that a deterministic algorithm exists which, given an error level $\zeta$ and a solution $y(x)\in{\cal R}(x)$ of the lower-level problem, computes an approximate solution $\tilde y(x,\zeta)$ to the lower-level problem.

Note also that point (a) of Assumption~\ref{ass:mapping} is more general than \cite[Assumption 2.1]{diouane2023inexact}, whereas point (b) is the same as \cite[Assumption 2.2]{diouane2023inexact}. 
\end{remark}

Which approach is suitable depends on the assumed cooperation level between decision makers; when ${\cal R}(x)$ is a singleton both coincide. We now state assumptions on $F$ and $\bar F$.

\begin{ass}\label{ass:Lip}\ \par
    \begin{itemize}
        \item[(i)] The upper-level function $F(x,y)$ is Lipschitz continuous w.r.t. $y$ with constant $L_F$.
        \item[(ii)] The function $\bar F(x) = F(x,y(x))$ is Lipschitz continuous with constant $L_{\bar F}$.
        \item[(iii)] The function $\bar F$ is bounded from below, i.e. $F_{low} > -\infty$ exists such that for all $x\in X$
  \[
  \bar F(x) \geq F_{low}.
  \]
    \end{itemize}
\end{ass}

\begin{remark}\label{Remark}
It is worth noting that points (i) and (ii) of Assumption \ref{ass:Lip} implicitly give some regularity assumption on the mapping $y(x)$.   
Furthermore, for all $x\in X$ and $\zeta >0$, we can write

\[
P(x,\zeta) \geq \bar F(x) - |P(x,\zeta)-\bar F(x)| 
\geq F_{low} - L_F\| \tilde y(x,\zeta)-y(x)\| \geq F_{low} - L_F\zeta,
\]
hence also the function P is bounded from below.

Furthermore, we note that Assumption~\ref{ass:Lip} is standard in the recent literature, see e.g. \cite[Assumptions 2.3--2.5]{diouane2023inexact}.
\end{remark}

We highlight that Assumptions \ref{ass:mapping} and \ref{ass:Lip} hold true throughout the whole paper.

\subsection{Notations and definitions}

Given $v \in \mathbb{R}^{n_x}$, subscripts denote either components ($v_k$) or sequence membership ($\{v_k\}_{k\in\mathbb{N}}$); when ambiguous, $(v)_i$ denotes the $i$-th component. For $a, b \in \mathbb{R}^{n_x}$, $\max\{a, b\}$ and $\min\{a, b\}$ are component-wise.
We denote by $S(0,1) = \{ x \in \mathbb{R}^{n_x} : \|x\| = 1 \}$ the unit sphere, by $[x]_{[l,u]} = \max\{l, \min\{u, x\}\}$ the projection onto $X = \{x \in \mathbb{R}^{n_x} : l_i \leq x_i \leq u_i\}$, and by $B_\delta(\bar{x}) = \{ x \in \mathbb{R}^{n_x} : \|\bar{x} - x\| \leq \delta \}$ the closed ball of radius $\delta$ centered at $\bar{x}$.

Since we are working in a (possibly) non-smooth context, we employ notions of stationarity based on Clarke's nonsmooth analysis. We recall the definitions of Clarke's generalized directional derivative and the Clarke-Jahn directional derivative, which account for the presence of  constraints.

\begin{definition}[Clarke and Clarke-Jahn generalized directional derivative]

Given a function $\bar F:\Re^{n_x}\to\Re$ Lipschitz continuous near a point $\bar x\in\Re^{n_x}$ and a direction $d\in\Re^{n_x}$, the Clarke generalized directional derivative is
\[
\bar F^{Cl}(\bar x,d) = \limsup_{\substack{x\to \bar x,\, t\downarrow 0}} \frac{\bar F(x+td)-\bar F(x)}{t}.
\]

The subdifferential (or generalized gradient) of $\bar F$ at $\bar x$ is the following set
\[
\partial \bar F(\bar x) = \{\xi \in \Re^{n_x}:\ \bar F^{Cl}(\bar x,d)\geq \xi^\top d,\ \forall\ d\in\Re^{n_x}\}.
\]
Given $X$, the Clarke-Jahn directional derivative is defined as
\[
\bar F^{\circ}(\bar x,d) = \limsup_{\substack{x\to \bar x,\, t\downarrow 0 \\ x\in X,\, x+td\in X}} \frac{\bar F(x+td)-\bar F(x)}{t}.
\]
\end{definition}

\begin{definition}[Dense sequence]
Let $K$ be an infinite subset of indices (possibly $K = \{0,1,\ldots\}$). 
The subsequence of normalized directions $\{d_k\}_{k \in K}$ is said to be 
\emph{dense in the unit sphere} $S(0,1)$ if, for any $\bar d \in S(0,1)$ and 
for any $\varepsilon > 0$, there exists an index $k \in K$ such that $\| d_k - \bar d \| \le \varepsilon $.
\end{definition}

\section{Bound-Constrained Upper-Level Problem}\label{sec:boundconstr}

In this section, we consider problem~\eqref{prob:upperlevel}, where the upper-level feasible set is defined solely by bound constraints. This setting serves as the starting point for our methodology, and the convergence results established here will form the foundation for the general framework presented in Section~\ref{sec:genconstr}.

Under Assumption \ref{ass:mapping}(a) the bilevel problem reduces to the following single-level reformulation:
\begin{equation}\label{prob:onlyboxconstraint}
  \min_{x\in X}\; \bar F(x).
\end{equation}

Before introducing our algorithmic framework, we recall the necessary stationarity notions for bound-constrained optimization in a derivative-free setting.

\begin{definition}[Cone of Feasible Directions]
Given \( X \) defined in~\eqref{prob:upperlevel}, the cone of feasible directions at \( x \) w.r.t.\ \( X \) is
\[
D(x) = \bigl\{ d \in \mathbb{R}^{n_x} :\, d_i \geq 0\ \text{if}\ x_i\!=\!l_i,\;
 d_i \leq 0\ \text{if}\ x_i\!=\!u_i,\;
 d_i \!\in\! \mathbb{R}\ \text{otherwise} \bigr\}.
\]
\end{definition}

\begin{definition}[Clarke and Clarke-Jahn stationarity]\label{def:puntostazionario}

Given problem~\eqref{prob:onlyboxconstraint}, a point $x^*\in X$ is Clarke-stationary when

\[
\bar F^{Cl}(x^*,d) = \max\{\xi^\top d:\ \xi\in\partial \bar F(x^*)\} \geq 0,\quad\forall\  d\in D(x^*).
\]

A point $x^*\in X$ is Clarke-Jahn-stationary when
\[
\bar F^{\circ}(x^*,d) \geq 0,\quad\forall\  d\in D(x^*).
\]

\end{definition}

Following the notion of Goldstein stationarity, reported for instance in \cite{Lin2022}, we define Goldstein stationarity for problem~\eqref{prob:onlyboxconstraint}.

\begin{definition}[\((\delta, \epsilon)\)-Goldstein stationarity]\label{def:Goldsteinstazionario}
Given $\delta, \epsilon$ strictly positive scalars, a point \( \bar{x} \in X \) is said to be a \((\delta, \epsilon)\)-Goldstein stationary point for problem~\eqref{prob:onlyboxconstraint} if there exist some \( \tilde{x} \in B_\delta(\bar{x}) \cap X \) and some \( \xi \in \partial \bar{F}(\tilde{x}) \) such that  
\[
\xi^\top d \geq -\epsilon, \quad \forall d \in D(\tilde{x}) \cap S(0,1).
\]

\end{definition}

\subsection{A Linesearch Derivative-free Method for Bilevel Programming}

Since exact evaluations of $\bar F(x) = F(x,y(x))$ would require solving the lower-level problem to optimality at each upper-level iteration, we instead work with the perturbed function $P(x,\zeta) = F(x,\tilde y(x,\zeta))$ defined in \eqref{def:functions}. This leads us to consider the following perturbed problem.
\begin{equation}\label{perturbed_prob:onlyboxconstraint}
  \min_{x\in X}\; P(x,\zeta).
\end{equation}

We now describe our Derivative-Free algorithm for Nonsmooth bilevel optimization problems with Lower-Level Adaptive accuracy (DFN-LLA). It solves the bilevel problem~\eqref{prob:onlyboxconstraint} through the perturbed reformulation~\eqref{perturbed_prob:onlyboxconstraint}. The distinguishing feature of our approach is the adaptive management of the lower-level accuracy through the mapping $\tilde y(x,\zeta)$.
The lower-level tolerance $\zeta$ starts at a relaxed value $\zeta_0$ and progressively decreases to $\bar\zeta\geq 0$ as the upper-level step sizes shrink, avoiding expensive high-accuracy lower-level solves when the upper-level iterate is far from optimality.\par

At each iteration of the algorithm, a search direction $d_k$ and (possibly) its opposite $-d_k$ are examined by
means of the {\it Projected Extrapolation} procedure. The Projected Extrapolation procedure computes a step $\alpha_k\geq 0$ and a direction $\tilde d_k$ (which is either $d_k$ or $-d_k$) given the current iterate $x_k$, the initial stepsize $\tilde\alpha_k$, the direction $d_k$, and the current precision level $\zeta_k$ for the lower-level problem. Specifically:
\begin{itemize}
    \item[(i)] if $\alpha_k > 0$, then point $\tilde x_k = [x_k+\alpha_k\tilde d_k]_{[l,u]}$ achieves sufficient decrease with respect to $x_k$ and the iteration is deemed successful; in this situation, the algorithm sets $\tilde\alpha_{k+1}=\alpha_k$, $\zeta_{k+1}=\zeta_k$ and the new point $x_{k+1}$ such that $P(x_{k+1},\zeta_k) \leq P(\tilde x_k,\zeta_k)$. This flexibility allows for additional local search or refinement strategies to be incorporated without affecting the convergence analysis, as long as non-increase of the objective is maintained (note that a feasible choice is $x_{k+1}=\tilde x_k$);

    \item[(ii)] if $\alpha_k = 0$, then sufficient decrease cannot be attained in either direction $\pm d_k$ with the stepsize $\tilde\alpha_k$ and the iteration is deemed unsuccessful; when this occurs, the algorithm shrinks the stepsize for the next iteration by setting $\tilde \alpha_{k+1} = \max\{\alpha_{\min},\theta\tilde\alpha_k\}$ with a saturation on $\alpha_{\min}$, defines the new iterate $x_{k+1}=x_k$ and checks the condition to trigger the update of the lower-level precision parameter $\zeta_k$. In particular, if $\tilde\alpha_{k+1}^3 < \zeta_k$ then $\zeta_{k+1}$ is reduced according to the rule $\zeta_{k+1} = \max\{\bar\zeta,\min\{\theta\zeta_k,\tilde\alpha_{k+1}^3\}\}$
    otherwise, $\zeta_{k+1}=\zeta_k$.
\end{itemize}
 
The updating rule for the precision parameter, especially at unsuccessful iterations, connects the stepsize with the precision parameter. 
This coupling between step size and accuracy ensures that we do not prematurely tighten the lower-level tolerance while the upper-level search is still making progress with larger steps. On the other hand, it requires using small tolerances if the objective function values are compared at very close points (small stepsize).
\par
As concerns the Projected Extrapolation procedure, this is where the main computations are carried out. The procedure accomplishes a twofold task. On the one hand, it determines whether $d_k$ (or $-d_k$) is a good descent direction, i.e., one along which sufficient decrease can be attained with respect to the initial point. On the other hand, if either $d_k$ or $-d_k$ is a good descent direction, it tries to enlarge the step as much as possible as long as sufficient decrease is guaranteed.
\par\smallskip
Algorithm DFN-LLA and the Projected Extrapolation procedure are reported in the boxes below.

\begin{algorithm}[!htbp]
\caption{DFN--LLA}
\small
\begin{algorithmic}[1]
\State \textbf{Input:}
$\theta \in (0,1)$,
$\sigma \in (0,1)$,
$\bar\zeta \ge 0$,
$x_0 \in X$,
$\{d_k\}$ with $d_k \in \mathbb{R}^{n_x}$ and $\|d_k\| = 1$ for all $k$.

\State $\alpha_{\min} \gets (\sigma \bar\zeta)^{1/3}$, 
Choose $\zeta_0 \ge \bar\zeta$, 
Choose $\tilde\alpha_0 \ge \zeta_0^{1/3}$

\For{$k = 0,1,2,\dots$}

    \State $(\alpha_k,\tilde d_k) \gets 
    \Call{ProjectedExtrapolation}{\tilde\alpha_k,x_k,d_k,\zeta_k}$

    \If{$\alpha_k = 0$} \Comment{unsuccessful iteration}

        \State $\tilde\alpha_{k+1} \gets 
        \max\{\alpha_{\min}, \theta \tilde\alpha_k\}$,
        $\tilde x_k \gets x_k$,
        $x_{k+1} \gets \tilde x_k$

        \If{$\tilde\alpha_{k+1}^3 < \zeta_k$}
            \State $\zeta_{k+1} \gets
            \max\{\bar\zeta,
            \min\{\theta \zeta_k,
            \tilde\alpha_{k+1}^3\}\}$
        \Else
            \State $\zeta_{k+1} \gets \zeta_k$
        \EndIf

    \Else \Comment{successful iteration}

        \State $\tilde\alpha_{k+1} \gets \alpha_k$,
        $\tilde x_k \gets [x_k + \alpha_k \tilde d_k]_{[l,u]}$,
        $\zeta_{k+1} \gets \zeta_k$

        \State Find $x_{k+1} \in X$ such that
        $P(x_{k+1},\zeta_k) \le P(\tilde x_k,\zeta_k)$.

    \EndIf

\EndFor

\State \textbf{Output:}
$\{x_k\}$, $\{\alpha_k\}$, $\{\tilde\alpha_k\}$, $\{\zeta_k\}$.

\end{algorithmic}
\end{algorithm}
\begin{algorithm}[!htbp]
\caption{ProjectedExtrapolation$(\tilde\alpha,x,d,\zeta)$}
\small
\begin{algorithmic}[1]
\State \textbf{Data:} $\gamma > 0$, $\delta \in (0,1)$
\State $\alpha \gets \tilde\alpha$

\If{$P([x+\alpha d]_{[l,u]},\zeta)
    \le P(x,\zeta) - \gamma \alpha^2$}
    \State $p^+ \gets d$
\ElsIf{$P([x-\alpha d]_{[l,u]},\zeta)
    \le P(x,\zeta) - \gamma \alpha^2$}
    \State $p^+ \gets -d$
\Else
    \State $\alpha \gets 0$
    \State \Return $\alpha, d$
\EndIf
\While{true} \Comment{extrapolation phase}
    \State $\beta \gets \alpha/\delta$
    \If{$P([x+\beta p^+]_{[l,u]},\zeta)
        > P(x,\zeta) - \gamma \beta^2$}
        \State \Return $\alpha, p^+$
    \EndIf
    \State $\alpha \gets \beta$
\EndWhile

\end{algorithmic}
\end{algorithm}

Note that, in the initialization of the algorithm, when $\bar\zeta = 0$ it is reasonable to select $\zeta_0 > 0$. Otherwise, we are actually requiring to solve the lower-level problem exactly at every iteration and this in turn makes the analysis with variable precision carried out below somewhat useless.

We remark that the lack of a stopping condition allows us to study the asymptotic convergence properties
of DFN-LLA. 

The following proposition guarantees that the algorithm DFN-LLA is well-defined meaning particularly that the Projected Extrapolation cannot infinitely cycle.

\begin{prop}\label{prop:nocycle_projectedLS}
The {\it Projected Extrapolation} cannot cycle indefinitely between Step 11 and Step 17.
\end{prop}

\textbf{Proof.} Let us consider the {\it Projected Extrapolation}. We proceed by contradiction, assuming that an infinite monotonically increasing sequence of positive numbers \(\{\beta_j\}\) exists such that  
\[
P([x + \beta_j p^+]_{[l,u]}, \zeta) \leq P(x, \zeta) - \gamma \beta_j^2.
\]
The above relation contradicts Assumption \ref{ass:Lip}(iii). $\hfill\Box$
\par\smallskip
Before delving further into the convergence analysis of the algorithm, we need to study the properties of the sequence of precision parameters $\{\zeta_k\}$. In particular, in the following proposition, we show that either the precision parameter stays fixed from a certain iteration on, or a special condition between iterations $k$ and $k-1$ is respected.
 
\begin{prop}\label{prop:boundzeta_complete} 
Let $\{\zeta_k\}$ and $\{\tilde\alpha_k\}$ be the sequences produced by algorithm DFN-LLA. 
\begin{itemize}
    \item[(i)] If ${\bar k}$ exists such that $\zeta_{\bar k} = \bar\zeta$, then $\zeta_k = \bar\zeta$, for all $k \geq \bar k$;
    \item[(ii)] otherwise, i.e.,\ $\zeta_{k} > \bar \zeta$ for all $k \geq 1$, then $\zeta_{k} \leq \min\{\zeta_{k-1},\tilde\alpha_{k}^3\}$.

\end{itemize}

\end{prop}
\textbf{Proof.} Recall that, by the instructions of the algorithm, $\zeta_{k+1}$ is defined according to the following rule:
\[
\zeta_{k+1} = \begin{cases}
\max\{\bar\zeta,\min\{\theta\zeta_k,\tilde\alpha_{k+1}^3\}\}
& \ \text{if}\ \alpha_k=0,\ \tilde\alpha_{k+1}^3< \zeta_k \\
\zeta_k & \ \text{if}\ \alpha_k=0,\ \tilde\alpha_{k+1}^3 \geq \zeta_k \\
\zeta_k & \ \text{if}\ \alpha_k>0. 
\end{cases}
\]
Now, we separately consider two cases (i) and (ii). 

Case (i) follows from the fact that $\zeta_{k+1} \leq \max\{\bar\zeta,\zeta_k\}$; hence, the sequence $\{\zeta_k\}$ is monotone non-increasing.

Now, we consider case (ii). We proceed by induction on the iteration index. From the initialization of the algorithm, we have that 
$\tilde\alpha_0 \geq \zeta_0^{1/3}$ and, by assumption, $\zeta_1>\bar\zeta$. Then, we have
\[
\zeta_1 = \begin{cases}
\min\{\theta\zeta_0,\tilde\alpha_1^3\}
\leq \min\{\zeta_0,\tilde\alpha_1^3\}& \ \text{if}\ \alpha_0=0,\ \tilde\alpha_1^3< \zeta_0 \\
\zeta_0 \leq \min\{\zeta_0,\tilde\alpha_1^3\}& \ \text{if}\ \alpha_0=0,\ \tilde\alpha_{1}^3 \geq \zeta_0 \\
\zeta_0 \leq \min\{\zeta_0,\tilde\alpha_1^3\}& \ \text{if}\ \alpha_0>0,   
\end{cases}
\]
where the last case follows considering that if $\alpha_0 > 0$, we have that $\tilde\alpha_1 = \alpha_0\geq\tilde\alpha_0$. So that $\tilde\alpha_1^3 \geq \tilde\alpha_0^3 \geq \zeta_0$. Then
\[
\zeta_1 = \zeta_0 = \min\{\zeta_0,\tilde\alpha_1^3\},\ \text{when}\ \alpha_0>0.
\]
Hence, when $k=0$, it is proved that
\[
\zeta_1 \leq \min\{\zeta_0,\tilde\alpha_1^3\}.
\]

Now, assuming that $\zeta_k\leq  \min\{\zeta_{k-1},\tilde\alpha_k^3\}$, we show that $\zeta_{k+1} \leq  \min\{\zeta_k,\tilde\alpha_{k+1}^3\}$ for all $k \geq1$ such that $\zeta_{k+1} > \bar \zeta$. 
\[
\zeta_{k+1} = \begin{cases}
\min\{\theta\zeta_k,\tilde\alpha_{k+1}^3\}
\leq \min\{\zeta_k,\tilde\alpha_{k+1}^3\}& \ \text{if}\ \alpha_k=0,\ \tilde\alpha_{k+1}^3< \zeta_k \\
\zeta_k \leq \min\{\zeta_k,\tilde\alpha_{k+1}^3\}& \ \text{if}\ \alpha_k=0,\ \tilde\alpha_{k+1}^3 \geq \zeta_k \\
\zeta_k \leq \min\{\zeta_k,\tilde\alpha_{k+1}^3\}& \ \text{if}\ \alpha_k>0,   
\end{cases}
\]
where, the last case follows by considering that $\tilde\alpha_{k+1} = \alpha_k \geq \tilde\alpha_{k}$ and $\tilde\alpha_k^3\geq\zeta_k$ by the inductive hypothesis. This concludes the proof.$\hfill\Box$

\subsection{Convergence analysis for the exact case (\(\bar\zeta = 0\))}

We first consider the case $\bar{\zeta} = 0$, where the algorithm eventually solves the lower-level exactly. The following proposition shows that the adaptive process drives $\zeta_k$ to zero.

\begin{prop}\label{prop:error_to_zero}
Let $\{\zeta_k\}$ be the sequence generated by the DFN-LLA algorithm when $\zeta_0 > \bar\zeta = 0$. Then,
\[
\lim_{k\rightarrow \infty} \zeta_k = 0.
\]
\end{prop}
\textbf{Proof.} By the instructions of the Algorithm and considering its initialization, the sequence of positive numbers $\{\zeta_k\}$ is monotone non-increasing. Then $\lim_{k\to\infty}\zeta_k = \hat\zeta\geq 0$. 
Assume by contradiction that $\hat\zeta>0$. Then since $\{\zeta_k\}$ is monotone non-increasing than 
\begin{equation}\label{diseq:absurd2}
    \zeta_k \geq \hat\zeta \ \text{for all}\ k \in \{0,1,2,\dots\}
\end{equation}
By Proposition~\ref{prop:boundzeta_complete}, we have that 
\begin{equation}\label{absurd1}
 \hat\zeta \leq \tilde\alpha_{k+1}^3,\qquad k \in \{0,1,2,\dots\}.
\end{equation}

Now, let us split the iteration sequence $\{0,1,2,\dots\}$, in two subsets $\cal S$ (i.e., successful iteration) and $\cal U$ (i.e., unsuccessful iteration) where:

\begin{itemize}
    \item[] \( \mathcal{S} = \{k \in \mathbb{N}_0 : ~\tilde{\alpha}_{k+1} = \alpha_k > 0, ~\zeta_{k+1} = \zeta_k \} \).
    \item[] \( \mathcal{U} = \mathcal{U}_1 \cup \mathcal{U}_2 \), where:
    \begin{itemize}
        \item[] $\mathcal{U}_1 = \{ k \in \mathbb{N}_0 : ~\tilde{\alpha}_{k+1} = \theta \tilde{\alpha}_{k},~\alpha_k = 0, ~\zeta_{k+1} = \zeta_k \}$
        \item[] $\mathcal{U}_2 = \{ k \in \mathbb{N}_0 : ~\tilde{\alpha}_{k+1} = \theta \tilde{\alpha}_{k},~\alpha_k = 0, ~\zeta_{k+1} = \min\{\theta\zeta_k,\tilde\alpha_{k+1}^3\} \leq \theta\zeta_k \}$.
    \end{itemize}
\end{itemize}

If $\mathcal{U}_2$ is infinite then for $k\in \mathcal{U}_2$, and sufficiently large, $\zeta_{k+1} < \hat\zeta$, contradicting \eqref{diseq:absurd2}. Then, ${\cal U}_2$ must be finite, hence, for $k$ sufficiently large, either $k\in\cal S$ or $k\in{\cal U}_1$. Thus, an index $\hat k$ exists such that for all $k\geq\hat k$, $\zeta_{k+1} = \zeta_k = \zeta_{\hat k} = \hat \zeta$.

If $\cal S$ is composed of a finite number of elements, this means that $k\in{\cal U}_1$ for all $k$ sufficiently large. In this case, for $k\in {\cal U}_1$, and sufficiently large, $\tilde\alpha_{k+1} < \hat\zeta^{1/3}$, contradicting (\ref{absurd1}).

Conversely, let us suppose that $\cal S$ is composed of an infinite number of elements. In this case, when $k \in \cal S$ and $k\geq\hat k$, we have
\begin{eqnarray}\label{last_1_1}
&&P(x_{k+1}, \hat \zeta)\leq P([x_k+\alpha_kd_k]_{[l,u]}, \hat \zeta)\leq \\ \nonumber
&&P(x_k,\hat \zeta)-\gamma(\alpha_k)^2\leq P(x_k,\hat \zeta)-\gamma \hat\zeta^{2/3},
\end{eqnarray}
where the last inequality is obtained by recalling \eqref{absurd1} and that $\alpha_k=\tilde\alpha_{k+1}$ when $k\in \cal S$. When $k\in \mathcal{U}_1$ and $k\geq\bar k$, by the instructions of the algorithm, we have that
\begin{equation}\label{last_1_2}
P(x_{k+1}, \hat \zeta)  = P(x_k,\hat \zeta).
\end{equation}
Hence, by \eqref{last_1_1} and \eqref{last_1_2}, the sequence $\{P(x_k,\hat \zeta)\}$ is monotonically non-increasing. Furthermore, by (\ref{last_1_1}), on the (infinite) subsequence corresponding to $\cal S$ the objective function $P$ decreases of a constant non-zero quantity. This would imply that the function $P(x,\hat \zeta) = F(x,\tilde y(x,\hat \zeta))$ is unbounded from below contradicting Assumption \ref{ass:Lip}(iii) and Remark \ref{Remark} thus concluding the proof.
$\hfill\Box$
\par\smallskip
We next show that the step sizes also vanish.

\begin{prop}\label{prop:alphastozero}
Let $\{\alpha_k\}$ and $\{\tilde\alpha_k\}$ be the sequences of actual and tentative step sizes generated by the DFN-LLA algorithm when $\zeta_0 > \bar\zeta = 0$. Then,
\[
\lim_{k \to \infty} \max\{\alpha_k, \tilde\alpha_k\} = 0.
\]
\end{prop}

\textbf{Proof.} First of all, note that without loss of generality $L_F$  of Assumption \ref{ass:Lip}(i) can be considered to be greater than 1. Then, since $\gamma\in(0,1)$, it always holds that $\gamma < 2L_F$. 

Now, recall, by Proposition \ref{prop:error_to_zero}, that $\{\zeta_k\}$ is such that 
\[
\lim_{k\to\infty}\zeta_k = 0.
\]
Let us define the following two set of iteration indices
\begin{itemize}
    \item[]${\cal S}=\{k \in \mathbb{N}_0:\alpha_k > 0\}$;
    \item[]${\cal U}=\{k  \in \mathbb{N}_0:\alpha_k = 0\}$.
\end{itemize}
When $k\in {\cal S}$, we have
\[
P(x_{k+1},\zeta_k) \leq P([x_k+\alpha_k\tilde d_k]_{[l,u]},\zeta_k) \leq P(x_k,\zeta_k) -\gamma(\alpha_k)^2.
\]
Moreover, by Assumption \ref{ass:Lip}(i) and Assumption \ref{ass:mapping}, we can write
\begin{align*}
P(x_{k+1},\zeta_k) &\geq F(x_{k+1},y(x_{k+1})) - |P(x_{k+1},\zeta_k) - F(x_{k+1},y(x_{k+1}))| \\
&\geq F(x_{k+1},y(x_{k+1})) - L_F\|\tilde y(x_{k+1},\zeta_k)-y(x_{k+1})\|\\
&\geq F(x_{k+1},y(x_{k+1})) - L_F \zeta_k
\end{align*}
and
\begin{align*}
P(x_{k},\zeta_k) &= F(x_k,\tilde y(x_k,\zeta_k)) \leq F(x_k,y(x_k)) + |P(x_k,\zeta_k) - F(x_k,y(x_k))| \\
&\leq F(x_k,y(x_k)) + L_F \zeta_k.
\end{align*}
Thus, by Proposition \ref{prop:boundzeta_complete} and recalling that $\alpha_k\geq\tilde\alpha_k$ when $k\in {\cal S}$, we can write
\begin{align*}
F(x_{k+1},y(x_{k+1})) &\leq F(x_k,y(x_k)) + 2L_F \zeta_k -\gamma\alpha_k^2\\
&\leq F(x_k,y(x_k)) + 2L_F \min\{\zeta_{k-1},\tilde\alpha_k^3\} -\gamma\alpha_k^2 \\
&\leq F(x_k,y(x_k)) + 2L_F \min\{\zeta_{k-1}, \alpha_k^3\} -\gamma\alpha_k^2.
\end{align*}

Now, as $k \to \infty$, if the set $\mathcal{S}$ is finite, then there exists an index $\hat{k}$ such that $k \in \mathcal{U}$ for all $k \geq \hat{k}$. Hence, we conclude that
\[
\lim_{k \to \infty} \alpha_k = 0.
\]

If instead the set $\mathcal{S}$ is infinite, since $\zeta_k \to 0$, we can consider sufficiently large indices $k \in \mathcal{S}$ such that
\[
  \zeta_{k-1} < \min\left\{\left(\frac{\gamma}{2L_F}\right)^3,1\right\} = \left(\frac{\gamma}{2L_F}\right)^3 < 1.
\]
Then, we consider the following cases.
\begin{enumerate}
    \item if $\min\{\zeta_{k-1},\alpha_k^3\}=\alpha_k^3<\zeta_{k-1}$. Then, $\alpha_k< \gamma/(2L_F)$ that is 
    $2L_F\alpha_k^3 < \gamma\alpha_k^2$ so that
    \[
    F(x_{k+1},y(x_{k+1})) < F(x_{k},y(x_{k}));
    \]
    \item if  $\min\{\zeta_{k-1},\alpha_k^3\}=\zeta_{k-1}$ we consider the following two subcases:
    \begin{itemize}
    \item[(i)] $\alpha_k \leq \gamma/(2L_F)\leq 1$. In this case $2L_F \min\{\zeta_{k-1},\alpha_k^3\} \leq 2L_F\alpha_k^3 \leq \gamma\alpha_k^2$, so that
        \[
    F(x_{k+1},y(x_{k+1})) \leq F(x_{k},y(x_{k}));
    \]
    \item[(ii)] $\gamma/(2L_F)< \alpha_k$. In this case we have $2L_F \min\{\zeta_{k-1},\alpha_k^3\} \leq 2L_F\zeta_{k-1} < \gamma^3/(2L_F)^2 < \gamma\alpha_k^2$, hence
        \[
    F(x_{k+1},y(x_{k+1}))< F(x_{k},y(x_{k})).
    \]
    \end{itemize}
\end{enumerate}

When $k\in {\cal U}$, i.e. the iterations of failure, we have $F(x_{k+1},y(x_{k+1})) = F(x_k,y(x_k))$. Thus, the sequence $\{F(x_k,y(x_k))\}$ is eventually a monotonically non-increasing sequence. By Assumption \ref{ass:Lip}(iii), we have that
\[
\lim_{k\to\infty}F(x_k,y(x_k)) = \bar F\geq F_{low}.
\]

When $k\in {\cal S}$, recalling that
\[
F(x_{k+1},y(x_{k+1}))  - F(x_k,y(x_k))  \leq  2L_F \min\{\zeta_{k-1}, \alpha_k^3\}  -\gamma(\alpha_k)^2 \leq 0
\]
and that 
\[
\lim_{k\to\infty,k\in {\cal S}} F(x_{k+1},y(x_{k+1}))  - F(x_k,y(x_k)) = 0
\]
we have
\begin{equation}\label{eq:lim1}
\lim_{k\to\infty,k\in {\cal S}}2L_F \min\{\zeta_{k-1}, \alpha_k^3\}  -\gamma(\alpha_k)^2 =0.
\end{equation}

Furthermore, since $\zeta_{k}> 0$, $\alpha_k\geq 0$ and $\zeta_{k-1}\to 0$, we also have that
\begin{equation}\label{eq:lim2}
\lim_{k\to\infty,k\in {\cal S}}2L_F \min\{\zeta_{k-1}, \alpha_k^3\} = 0.
\end{equation}
Then, by \eqref{eq:lim1} and \eqref{eq:lim2}, we have that $\lim_{k\to\infty,k\in {\cal S}}\alpha_k = 0$.
On the other hand, since $\alpha_k = 0$ when $k\in {\cal U}$, we trivially have that $\lim_{k\to\infty,k\in {\cal U}}\alpha_k = 0$,
so that we can conclude that
\begin{equation}\label{eq:alphaktozero}
\lim_{k\to\infty}\alpha_k = 0.
\end{equation}

Now, let us consider the sequence of tentative step sizes $\{\tilde\alpha_k\}$ and recall that, by the instructions of the algorithm,

\begin{itemize}
\item[(i)] when $k\in {\cal S}$, i.e. $\alpha_k > 0$, $\tilde\alpha_{k+1} = \alpha_k$;
\item[(ii)] when $k\in {\cal U}$, i.e. $\alpha_k=0$, $\tilde\alpha_{k+1} = \theta\tilde\alpha_k$.
\end{itemize}
Then, by \eqref{eq:alphaktozero}, we can write
\begin{equation}\label{eq:tildealphatozero1}
\lim_{k\to\infty,k\in {\cal S}}\tilde\alpha_{k+1} = 0.
\end{equation}
On the other hand, for every index $k\in {\cal U}$, let $m_k$ denote the biggest index such that $m_k\in {\cal S}$ and $m_k < k$. Then, we can write
\[
\tilde\alpha_{k+1} = \theta\tilde\alpha_k = \theta^{k-m_k}\alpha_{m_k}. 
\]
Now, when $k\to\infty$ and $k\in {\cal U}$, it results that either $m_k\to\infty$ (when ${\cal S}$ is infinite) or $k-m_k\to\infty$ (when ${\cal S}$ is finite). Thus, by the above relation, we can write
\begin{equation}\label{eq:tildealphatozero2}
\lim_{k\to\infty,k\in {\cal U}}\tilde\alpha_{k+1} = 0.
\end{equation}
Then, considering \eqref{eq:tildealphatozero1} and \eqref{eq:tildealphatozero2}, we can write
\[
\lim_{k\to\infty}\tilde\alpha_k = 0
\]
and the proof is concluded.
$\hfill\Box$
\par\smallskip
We now prove the main convergence result, after recalling two technical results. The first, from \cite[Lemma 2.6]{fasano2014linesearch}, concerns projected steps along feasible directions as the step size vanishes.

\begin{lemma}\label{lem:technicallemma1}
Let $\{x_k\}$, $\{d_k\}$ and $\{\eta_k\}$ be sequences of points, directions and scalars, respectively, with $\eta_k > 0$ for all $k$. Assume that for some infinite index set $K$, it results
\[
\lim_{k\to\infty,\,k\in K} x_k = \bar x, \qquad
\lim_{k\to\infty,\,k\in K} d_k = \bar d, \qquad
\lim_{k\to\infty,\,k\in K} \eta_k = 0 .
\]
with $\bar x\in X$, and $\bar d\in D(\bar x)$, $\bar d\neq 0$. Then
\begin{itemize}
    \item[(i)] for all $k\in K$ sufficiently large
    \[
        [x_k+\eta_k d_k]_{[l,u]} \neq x_k;
    \]
    \item[(ii)] the following limit holds:
    \[
       \lim_{k\to\infty, k\in K} v_k = \bar d
    \]
    where
    \[
    v_k = \frac{[x_k+\eta_k d_k]_{[l,u]} - x_k }{\eta_k}.
    \]
\end{itemize}
\end{lemma}
The next result (\cite[Proposition 2.3]{Lin2009}) states that feasible directions at $\bar x$ are also feasible for iterates sufficiently close to~$\bar x$.

\begin{prop}\label{lem:technicallemma2}
Given problem \eqref{prob:onlyboxconstraint}, let \( \{x_k\} \subset X \) for all \( k \) and \( \{x_k\} \to \bar{x} \) as \( k \to \infty \).  
Then, for \( k \) sufficiently large,  
\[
D(\bar{x}) \subseteq D(x_k).
\]
\end{prop}
Now, we are ready to state the main convergence result of this subsection.
\begin{Teorema}\label{prop:mainconv_perturbed}
Let $\{x_k\}$ be the sequence of iterates generated by the DFN-LLA algorithm when $\zeta_0 > \bar\zeta = 0$. Let $\bar{x}$ be any limit point of $\{x_k\}$, and let $K$ be an infinite set of indices such that
\[
\lim_{k \to \infty,\, k \in K} x_k = \bar{x}.
\]
If the subsequence $\{d_k\}_K$ of directions is dense in the unit sphere, then $\bar{x}$ is Clarke-Jahn stationary for problem~\eqref{prob:onlyboxconstraint}.
\end{Teorema}

\textbf{Proof.} We proceed by contradiction and assume $\bar x$ is not Clarke-Jahn stationary, i.e.\ a $\bar d\in D(\bar x)\cap S(0,1)$ exists such that $\bar F^\circ(\bar x;\bar d) < 0$.

By the instructions of the algorithm, we have that either
\begin{equation}\label{eq:fall1}
P([x_k+(\alpha_k/\delta)d_k]_{[l,u]}, \zeta_k) > P(x_k,\zeta_k) - \gamma (\alpha_k/\delta)^2,
\end{equation}
or
\begin{equation}\label{eq:fall2}
P([x_k+\tilde\alpha_kd_k]_{[l,u]}, \zeta_k) > P(x_k,\zeta_k) - \gamma (\tilde\alpha_k)^2.
\end{equation}
For every index $k\in K$, let us denote

\[\eta_k =
\begin{cases}
\displaystyle \frac{\alpha_k}{\delta} & \text{if \eqref{eq:fall1} holds}\\[0.4em]
\tilde\alpha_k & \text{if \eqref{eq:fall2} holds}
\end{cases}
\, \,  \text{ and} \, \,
v_k = \frac{[x_k+\eta_k d_k]_{[l,u]}-x_k}{\eta_k}.
\]

By Proposition \ref{prop:alphastozero}, we have that $\lim_{k\to\infty,k\in K}\eta_k = 0$. Furthermore, by the density assumption on $\{d_k\}_K$, an infinite subset $\bar K\subseteq K$ exists such that
\[
\lim_{k\to\infty,\,k\in \bar K} x_k = \bar x, \qquad
\lim_{k\to\infty,\,k\in \bar K} d_k = \bar d, \qquad
\lim_{k\to\infty,\,k\in \bar K} \eta_k = 0 .
\]
By considering Proposition \ref{lem:technicallemma2}, we have that, for sufficiently large $k$ and  $k\in K$,
$$x_k+\eta_k \bar d\in X.$$
Then, the definition of Clarke-Jahn generalized derivative, Assumption~\ref{ass:mapping} and points (i), (ii) of Assumption \ref{ass:Lip} imply:
\par\smallskip
{\footnotesize
\begin{align*}
\bar F^{\circ}(\bar x,\bar d) 
&= \limsup_{\substack{x_k\to \bar x,\, t\downarrow 0 \\ x_k\in X,\, x_k+t\bar d\in X}}
   \frac{\bar F(x_k+t\bar d)-\bar F(x_k)}{t} \\
&\geq \limsup_{k\to\infty,\,k\in\bar K}
   \frac{\bar F(x_k+\eta_k\bar d)-\bar F(x_k)}{\eta_k} \\
&= \limsup_{k\to\infty,\,k\in\bar K}
   \frac{\bar F(x_k+\eta_k\bar d) - \bar F(x_k+\eta_k v_k)
   + \bar F(x_k+\eta_k v_k)-\bar F(x_k)}{\eta_k} \\
&\geq \limsup_{k\to\infty,\,k\in\bar K}
   \frac{\bar F(x_k+\eta_k v_k)-\bar F(x_k)}{\eta_k}
   - L_{\bar F}\|v_k-\bar d\| \\
&= \limsup_{k\to\infty,\,k\in\bar K}
   \frac{\bar F(x_k+\eta_k v_k) - P(x_k+\eta_k v_k,\zeta_k)
   + P(x_k+\eta_k v_k,\zeta_k)-\bar F(x_k)}{\eta_k}
   - L_{\bar F}\|v_k-\bar d\| \\
&\geq \limsup_{k\to\infty,\,k\in\bar K}
   \frac{P(x_k+\eta_k v_k,\zeta_k)-\bar F(x_k)}{\eta_k}
   - L_F\frac{\zeta_k}{\eta_k} - L_{\bar F}\|v_k-\bar d\| \\
&= \limsup_{k\to\infty,\,k\in\bar K}
   \frac{P(x_k+\eta_k v_k,\zeta_k)-\bar F(x_k)
   + P(x_k,\zeta_k)-P(x_k,\zeta_k)}{\eta_k}
   - L_F\frac{\zeta_k}{\eta_k} - L_{\bar F}\|v_k-\bar d\| \\
&\geq \limsup_{k\to\infty,\,k\in\bar K}
   \frac{P(x_k+\eta_k v_k,\zeta_k)-P(x_k,\zeta_k)}{\eta_k}
   - 2L_F\frac{\zeta_k}{\eta_k} - L_{\bar F}\|v_k-\bar d\|.
\end{align*}
}
\par\smallskip
Concerning $\zeta_k/\eta_k$, by the definition of $\eta_k$ and the fact that $\zeta_k \leq\tilde\alpha_k^3$ (by Proposition~\ref{prop:boundzeta_complete}), we have that
\begin{itemize}
\item[] either $\eta_k = \tilde\alpha_k$ so that 
\begin{equation}\label{boundeta_1}
\frac{\zeta_k}{\eta_k} \leq \tilde\alpha_k^2\leq \eta_k^2;
\end{equation}
\item[] or $\eta_k = \alpha_k/\delta \geq \tilde\alpha_k/\delta$ so that
\begin{equation}\label{boundeta2}
\frac{\zeta_k}{\eta_k} \leq \delta\tilde\alpha_k^2\leq \delta^3\eta_k^2 \leq \eta_k^2.
\end{equation}
\end{itemize}

Then, from point (i) of Lemma \ref{lem:technicallemma1}, we have \( v_k \neq 0 \), for \( k \in \bar{K} \) and sufficiently large,  
so that relations \eqref{eq:fall1} and \eqref{eq:fall2} can be equivalently expressed as  
\[
P(x_k + \eta_k v_k,\zeta_k) > P(x_k, \zeta_k) - \gamma \eta_k^2.
\]
Hence

\[
\frac{  P(x_k+\eta_kv_k,\zeta_k)-P(x_k,\zeta_k)}{\eta_k} > -\gamma\eta_k.
\]

Finally, using the previous relation and recalling~\eqref{boundeta_1} and~\eqref{boundeta2}, we get that the following inequality
\[
  \bar F^{\circ}(\bar x,\bar d)\ge\limsup_{k\to\infty,k\in\bar K}\frac{  P(x_k+\eta_kv_k,\zeta_k)-P(x_k,\zeta_k)}{\eta_k} - 2L_F\frac{\zeta_k}{\eta_k} - L_{\bar F}\|v_k-\bar d\| ,
\]
which implies
\[
  \bar F^{\circ}(\bar x,\bar d)\ge \limsup_{k\to\infty,k\in\bar K}-\gamma \eta_k - 2L_F{\eta_k^2} - L_{\bar F}\|v_k-\bar d\|.
\]
Since $v_k\to\bar d$ and $\eta_k\to 0$ when $k\to\infty$, $k\in\bar K$ by point (ii) of Lemma \ref{lem:technicallemma1} we get that 
\[
  \bar F^{\circ}(\bar x,\bar d)\ge  0
\]
which is in contrast with $\bar F^\circ(\bar x,\bar d) < 0$ and concludes the proof.
$\hfill\Box$

\subsection{Convergence analysis for the inexact case ($\bar\zeta > 0$)}
We now consider $\bar{\zeta}>0$. Since the lower-level tolerance cannot vanish, the perturbed objective $P(x,\zeta)$ does not converge to $\bar F(x)$. Consequently, we cannot expect the algorithm to converge to exact Clarke-Jahn stationary points of problem~\eqref{prob:onlyboxconstraint}. Instead, we establish convergence to approximate stationary points characterized by the $(\delta, \epsilon)$-Goldstein stationarity condition which we have recalled in Definition~\ref{def:Goldsteinstazionario}. 
To facilitate the analysis, we partition the iteration indices into successful and unsuccessful sets, further subdividing the unsuccessful iterations based on whether the lower-level tolerance is updated. 

\begin{itemize}
\item[] 
$\mathcal{S} = \{k \in \mathbb{N}_0 : \tilde{\alpha}_{k+1} = \alpha_k > 0,\;
\zeta_{k+1} = \zeta_k \}$,
\item[]
$\mathcal{U} = \mathcal{U}_1 \cup \mathcal{U}_2$,
\end{itemize}
where
\[
\begin{aligned}
\mathcal{U}_1 &= \{ k \in \mathbb{N}_0 :
\tilde{\alpha}_{k+1} = \max\{\alpha_{\min},\theta\tilde\alpha_k\},\;
\alpha_k = 0,\;
\zeta_{k+1} = \zeta_k \}, \\
\mathcal{U}_2 &= \{ k \in \mathbb{N}_0 :
\tilde{\alpha}_{k+1} = \max\{\alpha_{\min},\theta\tilde\alpha_k\},\;
\alpha_k = 0,\;
\zeta_{k+1} = \max\{\bar\zeta,\min\{\theta\zeta_k,\tilde\alpha_{k+1}^3\}\} \}.
\end{aligned}
\]
Note that, when $k\in{\cal U}_1$, i.e. when the iteration is unsuccessful but we do not update the precision parameter, we have $\zeta_k\leq \tilde\alpha_{k+1}^3$.

The following proposition shows that the iteration sequence eventually stabilizes with the step size parameter reaching its minimum value $\alpha_{\min}$.

\begin{prop}\label{prop:Sfinite}
Let $\{x_k\}$ and $\{\tilde\alpha_k\}$ be the sequences of iterates and tentative step size parameters generated by Algorithm DFN-LLA when $\bar\zeta > 0$. Then the set $\cal S$ of successful iteration indices is finite, i.e., the sequence of iterates $\{x_k\}$ is eventually constant.  Furthermore, for $k$ sufficiently large, we have $\tilde\alpha_{k+1} = \tilde\alpha_k = \alpha_{\min}>0$.    
\end{prop}

\textbf{Proof.}
By the instructions of the Algorithm DFN-LLA and considering its initialization, the sequence of positive numbers $\{\zeta_k\}$ is monotone non-increasing. Then $\lim_{k\to\infty}\zeta_k = \hat\zeta\geq 0$. 

From the updating rule and the initialization of the algorithm, it must be $\hat \zeta \geq \bar \zeta = \frac{\alpha_{\min}^3}{\sigma} > 0.$ 
Hence, there exists an index \( \hat{k} \) such that $\zeta_k = \zeta_{\hat{k}} = \hat{\zeta}$ (if $\zeta_0 = \bar\zeta$, then $\hat k = 0$), and
\[
P(x_k, \zeta_k) = P(x_k, \hat{\zeta}),\text{ for all } k \geq \hat{k},
\]  
Then, by the instructions of the Projected Extrapolation method and by taking into account that $\alpha_k\geq\tilde\alpha_k\geq\alpha_{\min}$, we have at every  successful iterations $k\in\cal S,$ $ k \geq \hat{k}$ that
\[
P(x_{k+1},\hat\zeta) \leq P(x_k,\hat\zeta) -\gamma(\alpha_k)^2 \leq P(x_k,\hat\zeta) -\gamma(\alpha_{\min})^2.
\]
Hence, by Assumption \ref{ass:Lip}(iii), the number of successful iterations must be finite.  

Since $\cal S$ is finite, from a given $k$ sufficiently large on we have
\begin{equation}\label{eq:77}
\tilde\alpha_{k+1} = \max\{\alpha_{\min},\theta\tilde\alpha_k\}\geq\alpha_{\min}.    
\end{equation}
Let us suppose, by contradiction, that $\tilde\alpha_{k+1}>\alpha_{\min}$ for all $k$ sufficiently large. This, by \eqref{eq:77}, means that $\tilde\alpha_{k+1}=\theta\tilde\alpha_k$ for all $k$ sufficiently large. Then, since $\theta\in(0,1)$, $\lim_{k\to\infty}\tilde\alpha_k = 0$, which contradicts $\tilde\alpha_{k+1}>\alpha_{\min}>0$, concluding the proof.$\hfill\Box$
\par\smallskip
Proposition~\ref{prop:costantzeta} completes the asymptotic characterization by showing that the lower-level tolerance eventually reaches and remains at the target value $\bar{\zeta}$.

\begin{prop}\label{prop:costantzeta}
Let $\{\zeta_k\}$ be the sequence of error parameters generated by Algorithm DFN-LLA when $\bar\zeta > 0$. Then an index $\bar k \in \mathbb{N}_0$ exists such that $\zeta_k=\bar\zeta$ for $k\geq\bar k$.   
\end{prop}

\textbf{Proof.}
If $\zeta_0 = \bar \zeta$ then by Proposition \ref{prop:boundzeta_complete} the proof is ended. If $\zeta_0 > \bar \zeta$ then by the instructions of the Algorithm DFN-LLA and considering its initialization, the sequence of positive numbers $\{\zeta_k\}$ is monotone non-increasing. Then $\lim_{k\to\infty}\zeta_k = \hat\zeta\geq 0$. 

From the updating rule and the initialization of the algorithm, it must be $\hat \zeta \geq \bar \zeta = \frac{\alpha_{\min}^3}{\sigma}.$ Assume by contradiction that $\hat\zeta>\bar \zeta$.
Then, since $\{\zeta_k\}$ is monotone non-increasing, we have that
\begin{equation}\label{diseq:absurd3}
    \zeta_k \geq \hat\zeta >\bar\zeta \ \text{for all}\ k \in \{0,1,2,\dots\}.
\end{equation}

This in turn implies that, by Proposition~\ref{prop:boundzeta_complete} and considering that $\hat \zeta  > \bar \zeta$,
\begin{equation}\label{absurd16}
 \hat\zeta \leq \tilde\alpha_{k+1}^3,\qquad \text{for all}\ k \in \{0,1,2,\dots\}.
\end{equation}

If $\mathcal{U}_2$ is infinite, for $k\in \mathcal{U}_2$ sufficiently large, $\zeta_{k+1} = \min\{\theta\zeta_k,\tilde\alpha_{k+1}^3\} \leq \theta\zeta_k$, so $\zeta_{k+1} < \bar\zeta$, contradicting \eqref{diseq:absurd3}. Hence $\mathcal{U}_2$ is finite, and by Proposition~\ref{prop:Sfinite}, $k\in\mathcal{U}_1$ for $k$ large enough. Thus $\hat k$ exists with $\zeta_k = \hat \zeta$ for all $k\geq\hat k$. 

In this case, by Proposition \ref{prop:Sfinite}, for $k\in {\cal U}_1$, and sufficiently large, $\tilde\alpha_{k+1} = \alpha_{\min}$. Thus, from \eqref{absurd16} we have:
\begin{equation}\label{absurd17}
\hat\zeta \leq \alpha_{\min}^{3}
\end{equation}
Now, considering that \( \alpha_{\min} = (\sigma \bar{\zeta})^{1/3} \) with \( \sigma \in (0,1) \), it follows that \( \bar{\zeta} = \alpha_{\min}^3/\sigma > \alpha_{\min}^3 \). Hence, we obtain \( \hat{\zeta} > \bar{\zeta} > \alpha_{\min}^3 \), which contradicts (\ref{absurd17}), and concludes the proof.  $\hfill\Box$
\par\smallskip
Now, we state the main convergence result of this subsection, namely that the sequence $\{x_k\}$ converges to a $(\delta,\epsilon)$-Goldstein stationary point.
 
\begin{Teorema}\label{prop:convGoldstainstaz}
Assume that $\{d_k\}$ is dense in the unit sphere. Then the sequence \(\{x_k\}\) generated by Algorithm DFN-LLA when $\bar\zeta > 0$ is eventually constant, with the unique limit point being \((\delta, \epsilon)\)-Goldstein stationary for Problem \eqref{prob:onlyboxconstraint}, with
\[
\epsilon = 4\gamma\alpha_{\min} + \frac{8L_F \bar\zeta}{\alpha_{\min}}  \quad \text{and} \quad \delta = \alpha_{\min}.
\]
\end{Teorema}
\textbf{Proof.} First, \(\{x_k\}\) is eventually constant, as shown in Proposition \ref{prop:Sfinite}. Let \(\bar{x}\) be the unique limit point. By the stepsize updating rule, every iteration must be unsuccessful with \(\tilde\alpha_k = \alpha_{\min}\) for \(k\) sufficiently large. 
Hence, we can write
\[
P(\bar x,\bar\zeta) < P([\bar x+\alpha_{\min} d_k]_{[l,u]},\bar\zeta) + \gamma\alpha_{\min}^2.
\]
By Assumption~\ref{ass:Lip}, we have
\[
\bar F(\bar x)  <  \bar F([\bar x+\alpha_{\min} d_k]_{[l,u]}) + \gamma\alpha_{\min}^2 + 2 L_F\bar \zeta 
\]
and, for every $d$ such that $\|d\| = \alpha_{\min}$, we can write
\begin{eqnarray*}
\bar F(\bar x) &<& \bar F([\bar x+\alpha_{\min} d_k]_{[l,u]}) - \bar F([\bar x+d]_{[l,u]}) + \bar F([\bar x+d]_{[l,u]}) + \gamma\alpha_{\min}^2 + 2 L_F\bar\zeta\\
&\leq& \bar F([\bar x+d]_{[l,u]}) + L_{\bar F} \left\|[\bar x+\alpha_{\min} d_k]_{[l,u]} - [\bar x+d]_{[l,u]}\right\|+ \gamma\alpha_{\min}^2 + 2 L_F\bar\zeta \\
&\leq& \bar F([\bar x+d]_{[l,u]}) + L_{\bar F} \left\|\alpha_{\min} d_k - d\right\|+ \gamma\alpha_{\min}^2 + 2 L_F\bar\zeta
\end{eqnarray*}

Then, by the density assumption on $\{d_k\}$, a subsequence $\{x_k\}_K$ exists such that:
\begin{eqnarray*}
 && \lim_{k\to\infty,k\in K} \alpha_{\min}d_k = d. 
\end{eqnarray*}
Therefore, for $k$ sufficiently large and $k\in K$, we have:
\[
L_{\bar F}\|\alpha_{\min}d_k - d\|\leq \gamma\alpha_{\min}^2 + 2 L_F\bar \zeta,
\]
so that we can thus write, for $k\in K$ sufficiently large and for every  $d$ such that $\|d\| = \alpha_{\min}$,
\begin{equation}\label{app1}
\bar F(\bar x) < \bar F([\bar x+ d]_{[l,u]}) + 2(\gamma\alpha_{\min}^2 + 2 L_F\bar \zeta).
\end{equation}
In particular, we have
\begin{eqnarray}
&&\label{app2}
\bar F(\bar x) < \bar F(\bar x + d) + 2\gamma\alpha_{\min}^2 + 4 L_F\bar \zeta
\end{eqnarray}
for every $d$ such that $\|d\|=\alpha_{\min}$ and $\bar x+d\in X$.

Now, we define the function 
$$\tilde G_{\bar x}(d) = \bar F(\bar x+d) + \left(2\gamma + 4\dfrac{L_F\bar\zeta}{\alpha_{\min}^2}\right)\|d\|^2.$$ From  \eqref{app2} we have that
\begin{equation}\label{diseq:nonmin}
    \tilde G_{\bar x}(0) < \tilde G_{\bar x}(d) 
\end{equation}

for all $d$ such that $\|d\|=\alpha_{\min}$ and $\bar x+d\in X$.
Now, consider the following problem
\begin{equation}\nonumber
 \begin{array}{ll}
  \displaystyle \min_{d}  & \tilde G_{\bar x}(d) \\
   s.t. & \|d\|\le \alpha_{\min} ,\\
   & \bar x+d\in X.
 \end{array}
\end{equation}
By compactness of the feasible region and continuity of the function $\tilde G_{\bar x}$, the problem admits a global minimum point $\bar d$. By \eqref{diseq:nonmin}, $\bar d$ is such that $\|\bar d\|< \alpha_{\min}$.  Hence, $\bar d$ is a local minimum point, hence stationary, of the problem:  
\begin{equation}\label{app3}
 \begin{array}{ll}
  \displaystyle \min_{d}  & \tilde G_{\bar x}(d) \\
   s.t. & \bar x+d\in X.\\
 \end{array}
\end{equation}
This implies the vector $\bar d$ satisfies Definition (\ref{def:puntostazionario}), namely vector  $\xi\in\partial\tilde G_{\bar x}(\bar d)$ exists such that
\[
\xi^\top s \geq 0,\quad\forall\  s\in \bar D(\bar d).
\]
where,
\[
\bar D(\bar d) = \bigl\{ s \in \mathbb{R}^{n_x} :\, s_i \geq 0\ \text{if}\ \bar x_i+\bar d_i = l_i,\;  s_i \leq 0\ \text{if}\ \bar x_i+\bar d_i= u_i,\;  s_i \in \mathbb{R}\ \text{otherwise} \bigr\},
\]
which is equivalent to $\bar D(\bar d)=D(\bar x+\bar d)$. By taking into account that
\[
\partial\tilde G_{\bar x}(\bar d) = \partial \bar F(\bar x+\bar d) + \left(4\gamma + 8 \dfrac{L_F\bar\zeta}{\alpha_{\min}^2}\right)\bar d
\]
we get that a vector  $\bar\xi\in\partial \bar F(\bar x+\bar d)$ exists such that
\[
\bar \xi^\top s \geq -\left(4\gamma + 8 \dfrac{L_F\bar\zeta}{\alpha_{\min}^2}\right)\bar d^Ts,\quad\forall\  s\in D(\bar x+\bar d).\]
Finally, recalling that $\|\bar d\|<\alpha_{\min}$, we obtain that $\bar x+\bar d \in B_{\alpha_{\min}}(\bar x) \cap X$ and  $\bar\xi\in\partial \bar F(\bar x+\bar d)$ exist such that
\[
\bar \xi^\top \frac{s}{\| s\|} \geq -\left(4\gamma \alpha_{\min} + 8 \dfrac{L_F\bar\zeta}{\alpha_{\min}}\right),\quad\forall\  \frac{s}{\| s\|}\in D(\bar x+\bar d).
\] 
This proves the theorem.
$\hfill\Box$
\par\smallskip
Theorem~\ref{prop:convGoldstainstaz} provides a quantitative estimate of the stationarity quality of the limit point when the lower-level tolerance is bounded away from zero. The parameter $\epsilon$ depends on $\bar{\zeta}$ in two distinct ways: directly through the term $L_F \bar{\zeta} / \alpha_{\min}$, and indirectly through $\alpha_{\min} = (\sigma \bar{\zeta})^{1/3} = \delta$. Indeed, one has 
\[
\epsilon = 4\gamma (\sigma \bar{\zeta})^{1/3} + \frac{8 L_F \bar{\zeta}^{2/3}}{\sigma^{1/3}}.
\]
This relation shows that decreasing $\bar{\zeta}$ by three orders of magnitude (e.g., from $10^{-6}$ to $10^{-9}$) yields an improvement of roughly one order of magnitude in the stationarity measure~$\epsilon$, reflecting the cube-root dependence on~$\bar{\zeta}$.

\section{Handling of General Inequality Upper-Level Constraints}\label{sec:genconstr}

We now extend the framework to handle general nonlinear inequality constraints at the upper-level via an exact penalty function approach.

Under Assumption \ref{ass:mapping}(a), problem \eqref{prob:upperlevel} is equivalent to the following
\begin{equation}\label{prob:generalconstraint}
  \min_{x\in X}\; \bar F(x) \quad \text{s.t.}\quad \bar H(x)\leq 0.
\end{equation}

To establish convergence results analogous to those obtained in the bound-constrained case, we require additional regularity conditions on the constraint. The following assumption extends Assumption~\ref{ass:Lip} to incorporate the constraint regularity.

\begin{ass}\label{ass:Lip_1}\ \par
    \begin{itemize}
        \item[(i)] The upper-level constraints $H_i(x,y)$ are Lipschitz continuous w.r.t. $y$ with constants $L_{H_i}$$,\ i=1,\dots,m$.
        \item[(ii)] The functions $\bar H_i(x)$ are Lipschitz continuous with constant $L_{\bar H_i},$$\,i=1,\dots,m$.
    \end{itemize}
    
Let us denote \( L = L_F + \frac{1}{\rho} \left(\max_{1\le i\le m} L_{H_i}\right) \) and  \( \bar L = L_{\bar F} + \frac{1}{\rho} \left(\max_{1\le i\le m} L_{\bar H_i}\right) \)
\end{ass}

The following definitions extend the stationarity concepts from Section~\ref{sec:boundconstr} to accommodate the presence of general inequality constraints. 

\begin{definition}[Clarke-KKT stationary]
Given problem~\eqref{prob:generalconstraint}, a feasible point \(\bar{x}\) is Clarke-KKT stationary for~\eqref{prob:generalconstraint} if multipliers \(\bar{\lambda}_1, \ldots, \bar{\lambda}_m \in \mathbb{R}\) exist, with  
\[
\bar{\lambda}_i \geq 0 \quad \text{and} \quad \bar{\lambda}_i \bar H_i(\bar{x}) = 0,\quad \bar H_i(\bar{x}) \leq 0 \quad \forall i = 1, \ldots, m,
\]  
such that for every \(d \in D(\bar{x})\),  
\[
\max\left\{d^\top \xi : \xi \in \partial \bar F(\bar{x}) + \sum_{i=1}^m \bar{\lambda}_i \partial \bar H_i(\bar{x})\right\} \geq 0.
\]
\end{definition}

\begin{definition}[\((\delta, \epsilon)\)-Goldstein-KKT stationary]\label{def:golsteinstazionario_kkt}
A feasible point \(\bar x\) is called a \((\delta, \epsilon)\)-Goldstein-KKT stationary point for problem~\eqref{prob:generalconstraint},  if, for some feasible $\tilde x \in B_\delta(\bar x)$, multipliers $\lambda_i\in\Re$, $i=1,\dots,m$, exists such that
    \[
    \lambda_i\geq 0,\ \lambda_i \bar H_i(\tilde x) = 0,\ \bar H_i(\tilde x)\leq 0,\ i=1,\dots,m,
    \]
    and $\xi\in \partial \bar F(\tilde x) + \sum_{i=1}^m\lambda_i\partial \bar H_i(\tilde x)$ such that
\[
\xi^\top d \geq -\epsilon,\quad \forall\ d\in D(\tilde x) \cap S(0,1).
\]
\end{definition}

Given problem~\eqref{prob:generalconstraint}, we introduce the following penalty function which allows us to manage the presence of the upper-level inequality constraints.
\[
Q_{\rho}(x) = \bar F(x) + \frac{1}{\rho}\max\{0,\bar H_1(x),\bar H_2(x),\dots,\bar H_m(x)\}. 
\]
Then, we can rewrite~\eqref{prob:generalconstraint} as
\begin{equation}
\label{prob:penalized_upperlevel}
  \min_{x\in X}\; Q_{\rho}(x) \quad 
\end{equation}

In \cite{fasano2014linesearch}, exactness properties of the function $Q_{\rho}(x)$ have been proved under the Extended Mangasarian-Fromovitz Constraint Qualification (EMFCQ) reported below.

\begin{ass}\label{ass:emfcq}
Given problem~\eqref{prob:generalconstraint}, for any $x\in X\setminus\ int\{x\in\Re^{n_x}:\ \bar H(x)\leq 0\}$ a direction $d\in D(x)$ exists such that
\[
(\xi^{\bar H_i})^\top d < 0
\]
for all $\xi^{\bar H_i}\in\partial \bar H_i(x)$, $i\in\{1,\dots,m:\ \bar H_i(x)\geq 0\}$.
\end{ass}
In particular, as stated in the following proposition from \cite{fasano2014linesearch}, it holds that a $\rho^*>0$ exists such that for all $\rho\in(0,\rho^*]$, every Clarke-Jahn-stationary point of problem~\eqref{prob:penalized_upperlevel} is Clarke-KKT-stationary for problem~\eqref{prob:generalconstraint}.  

\par\smallskip
We emphasize that Assumptions \ref{ass:Lip_1} and \ref{ass:emfcq} are always satisfied throughout this section.

\begin{prop}\label{prop:equivalenza}
Given problem \eqref{prob:generalconstraint} and considering problem \eqref{prob:penalized_upperlevel}, a threshold value $\varrho > 0$ exists such that for every $\rho \in (0, \varrho]$, every Clarke–Jahn stationary point $\bar{x}$ of problem \eqref{prob:penalized_upperlevel} is Clarke-KKT stationary for problem \eqref{prob:generalconstraint}.
\end{prop}
\textbf{Proof.} See \cite[Proposition 3.6]{fasano2014linesearch}. $\hfill\Box$
\par\smallskip
Proposition~\ref{prop:equivalenza} establishes the fundamental equivalence between stationarity for the penalized problem and KKT stationarity for the original constrained problem, provided the penalty parameter is sufficiently small. This result forms the theoretical basis for our algorithmic approach: instead of solving problem~\eqref{prob:generalconstraint} directly, we can solve the bound-constrained penalized problem~\eqref{prob:penalized_upperlevel} using the derivative-free algorithm developed in the previous section. However, since the lower-level problem is solved only approximately, we introduce the perturbed penalty function as
\[
R_{\rho}(x,\zeta) = F(x,\tilde y(x,\zeta)) + \frac{1}{\rho}\max\{0, H_1(x, \tilde y(x, \zeta)),\dots, H_m(x, \tilde y(x, \zeta))\}. 
\]
To solve problem~\eqref{prob:generalconstraint}, we apply Algorithm DFN-LLA with $P(x,\zeta)$ replaced by $R_\rho(x,\zeta)$.  Note that Assumption~\ref{ass:Lip}(iii) ensures that $R_\rho(x,\zeta)$ is bounded from below, while Assumption~\ref{ass:Lip_1} extends the regularity properties to both the constraints of the formulation and the penalized function $R_\rho(x,\zeta)$. 
Consequently, the convergence analysis developed in Section~\ref{sec:boundconstr} can be readily adapted to this more general setting. We can therefore state the following proposition.

\begin{prop}\label{prop:conv_exact_constrained}
Let $\{x_k\}$ be the sequence of points generated by the Algorithm DFN-LLA when $\zeta_0>\bar\zeta = 0$. Let $\bar x$ be any limit point of $\{x_k\}$ and $K$ an infinite subset of indices such that
\[
\lim_{k\to\infty,k\in K} x_k = \bar x.
\]
If the subsequence $\{d_k\}_K$ of directions is dense in the unit sphere,
then a threshold value $\varrho > 0$ exists such that for every $\rho \in (0, \varrho]$, $\bar x$ is Clarke-KKT stationary for problem \eqref{prob:generalconstraint}.
\end{prop}
\textbf{Proof.} The proof directly follows from Theorem \ref{prop:mainconv_perturbed} and Proposition \ref{prop:equivalenza}. $\hfill\Box$
\par\smallskip
We now consider the case $\bar\zeta > 0$.

\begin{Teorema}\label{prop:conv_inexact_constrained}
Assume that \(\{d_k\}\) is dense in the unit sphere. Then a threshold value $\varrho > 0$ exists such that for every $\rho \in (0, \varrho]$, the sequence \(\{x_k\}\) generated by Algorithm DFN-LLA when $\bar\zeta>0$ is eventually constant, with the unique limit point being \((\delta, \epsilon)\)-Goldstein-KKT stationary for problem \eqref{prob:generalconstraint}, with
\[
\epsilon =   4\gamma\alpha_{\min} + 8\frac{L \bar\zeta}{\alpha_{\min}}\quad \text{and} \quad \delta = \alpha_{\min}.
\]
\end{Teorema}

\textbf{Proof.} 
The proof follows the same structure as that of Theorem~\ref{prop:convGoldstainstaz}, with $R_\rho$ replacing~$P$, $Q_\rho$ replacing~$\bar F$, and the Lipschitz constants $L$, $\bar L$ (from Assumption~\ref{ass:Lip_1}) replacing $L_F$, $L_{\bar F}$, respectively.
By Proposition~\ref{prop:costantzeta}, $\{x_k\}$ is eventually constant with limit~$\bar{x}$, and $\tilde\alpha_k = \alpha_{\min}$ for $k$ sufficiently large. Repeating the argument leading to~\eqref{app2} with the above substitutions, we obtain
\begin{equation}\label{app2_vinc}
Q_{\rho}(\bar x) < Q_{\rho}(\bar x + d) + 2\gamma\alpha_{\min}^2 + 4 L\bar \zeta
\end{equation}
for all $d$ such that $\|d\|=\alpha_{\min}$ and $\bar x+d\in X$.
Defining
$\tilde Q_{\rho,\bar x}(d) = Q_{\rho}(\bar x+d) + 2\left(\gamma + 2\frac{L\bar\zeta}{\alpha_{\min}^2}\right)\|d\|^2$,
the same argument as in Theorem~\ref{prop:convGoldstainstaz} yields a point $\bar d$ with $\|\bar d\| < \alpha_{\min}$ that is a local minimizer of: 
\begin{equation}\label{app3_vinc}
 \begin{array}{ll}
  \displaystyle \min_{d}  & \tilde Q_{\rho,\bar x}(d) \\
   s.t. & \bar x+d\in X.\\
 \end{array}
\end{equation}
This implies the vector $\bar d$ satisfies Definition (\ref{def:puntostazionario}), namely vector  $\xi\in\partial\tilde Q_{\rho, \bar x}(\bar d)$ exists such that
\[
\xi^\top s \geq 0,\quad\forall\  s\in \bar D(\bar d),
\]
reminding that $\bar D(\bar d)=D(\bar x+\bar d).$

Recalling the exactness properties of the penalty function, provided that the penalty parameter is sufficiently small, $\bar x+\bar d$ is a Clarke-KKT stationary point of the problem
\[
\begin{array}{ll}
\displaystyle \min_{d} & \bar F(\bar x+d) + 2\left(\gamma + 2\dfrac{L\bar\zeta}{\alpha_{\min}^2}\right)\|d\|^2 \\[12pt]
\text{s.t.} & \bar H(\bar x+d) \leq 0, \,\bar x+d \in X.
\end{array}
\]
that is, multipliers \(\bar{\lambda}_1, \ldots, \bar{\lambda}_m \in \mathbb{R}\) exist, with  
\[
\bar{\lambda}_i \geq 0 \quad \text{and} \quad \bar{\lambda}_i \bar H_i(\bar{x}+ \bar d) = 0, \quad \bar H_i(\bar{x}+\bar d) \leq 0\quad \forall i = 1, \ldots, m,
\]  
such that for every \(s \in D(\bar{x}+\bar d)\),  
\[
\max\left\{s^\top \xi : \xi \in \partial \bar F(\bar{x}+\bar d) + \left(4\gamma + 8 \dfrac{L\bar\zeta}{\alpha_{\min}^2}\right)\bar d+\sum_{i=1}^m \bar{\lambda}_i \partial \bar H_i(\bar{x}+\bar d)\right\} \geq 0.
\]

we get that a vector  $\bar\xi\in\partial \bar F(\bar{x}+\bar d) +\sum_{i=1}^m \bar{\lambda}_i \partial \bar H_i(\bar{x}+\bar d)$ exists such that
\[
\bar \xi^\top s \geq -\left(4\gamma + 8 \dfrac{L\bar\zeta}{\alpha_{\min}^2}\right)\bar d^Ts,\quad\forall\  s\in D(\bar x+\bar d).\]
Finally, recalling that $\|\bar d\|<\alpha_{\min}$, we obtain that $\bar x+\bar d \in B_{\alpha_{\min}}(\bar x)$
and  $\bar\xi\in\partial \bar F(\bar{x}+\bar d) +\sum_{i=1}^m \bar{\lambda}_i \partial \bar H_i(\bar{x}+\bar d)$  exists such that
\[
\bar \xi^\top \frac{s}{\| s\|} \geq -\left(4\gamma \alpha_{\min} + 8 \dfrac{L\bar\zeta}{\alpha_{\min}}\right),\quad\forall\  \frac{s}{\| s\|}\in D(\bar x+\bar d).
\] 
This proves the theorem. $\hfill\Box$
\par\smallskip
Theorem~\ref{prop:conv_inexact_constrained} shows that Algorithm DFN-LLA converges to an approximate Goldstein-KKT stationary point, with approximation quality controlled by the lower-level limit tolerance. The bound $\epsilon$ depends on the combined Lipschitz constant $L$, accounting for both objective and constraint. As before, smaller $\bar\zeta$ yields higher-quality approximate stationary points, with $\alpha_{\min} = (\sigma\bar\zeta)^{1/3}$.

\section{Numerical Results}\label{sec:numerical}

In this section, a computational study is conducted to assess the performance of the proposed algorithm and its variants.

\subsection{Test problems collection}

The algorithms were evaluated on 160 small-scale bilevel optimization problems selected from the BOLIB MATLAB library \cite{Zhou2020}. Starting from the original 173 problems, instances with $n_x \geq 1$ were retained, using the default starting points provided in the library. All algorithms were implemented in Python 3.12.3. The distribution of test problems with respect to dimension and constraint structure is reported in Table~\ref{tab:nvars_count}. The source code is publicly available at \href{https://github.com/EdoardoCesaroni/DFN-LLA}{\texttt{github.com/EdoardoCesaroni/DFN-LLA}}.

All numerical experiments were performed on a virtual machine running Ubuntu 20.04.6 LTS. The computational environment was based on a server equipped with two NUMA sockets and an Intel Xeon Gold 6252N CPU @ 2.30GHz processor, providing a total of 16 logical cores. The system was endowed with 512 GB of RAM. The implementations relied mainly on NumPy for numerical linear algebra and computing operations. No GPU acceleration was employed in the computational runs.

\begin{table}[!htbp]
\centering
\small
\begin{tabular}{ccccccccccccc}
\toprule
$n_x + n_y$ & 2 & 3 & 4 & 5 & 6 & 8 & 9 & 10 & 11 & 12 & 20 & 21 \\
\midrule
\# problems   & 69 & 29 & 29 & 10 & 4 & 3 & 1 & 7 & 1 & 1 & 2 & 4 \\
\bottomrule
\end{tabular}
\par\vspace{4pt}
\begin{tabular}{ccc}
\toprule
\textbf{Bound-constr.} & \textbf{General constr.} & \textbf{Unconstrained} \\
\midrule
78 & 51 & 31 \\
\bottomrule
\end{tabular}
\caption{Distribution of problems by total number of variables (top) and constraint type (bottom).}
\label{tab:nvars_count}
\end{table}

\subsection{Performance evaluation methodology}
To evaluate and compare the algorithms, the procedure for constructing performance and data profiles proposed in \cite{More2009} has been adapted to the bilevel optimization setting. The main modification concerns the definition of feasibility and convergence criteria, which must simultaneously account for the upper-level objective function and the KKT residuals of the lower-level problem. For each lower-level solution $\tilde y(x,\zeta)$ associated with an upper-level point $x$, a comprehensive KKT residual is computed:
{\small
\[
\text{KKT}_{\mathrm{res}} =
\max\left\{
\left\| \nabla_y \mathcal{L}\big(x,\tilde y(x,\zeta),\lambda\big) \right\|_{\infty},
\;
\left\| \max\{-\lambda,\, g\big(x,\tilde y(x,\zeta)\big)\} \right\|_{\infty},
\;
\left\| g\big(x,\tilde y(x,\zeta)\big) \right\|_{\infty}
\right\}
\]
} where $\nabla_y \mathcal{L}$ denotes the gradient of the lower-level Lagrangian with respect to $y$, and $\lambda$ are the Lagrange multipliers associated with the constraints $g(x,y) \le 0$ (including both general and bound constraints, for notational simplicity).

A point $(x, \tilde y(x,\zeta))$ is deemed valid if both of the following conditions are satisfied:
\begin{itemize}

\item[-] The upper-level violation is sufficiently small: $\left\| H\big(x, \tilde y(x,\zeta)\big) \right\|_{\infty} \le 10^{-4}$.

\item[-] The lower-level KKT residual achieves the prescribed accuracy:
$\text{KKT}_{\mathrm{res}} \le \bar{\zeta}$.

\end{itemize}
If at least one of these conditions is violated, the corresponding objective value is set to $+\infty$ for the purpose of constructing the profiles. This ensures that only points satisfying both feasibility and lower-level optimality requirements are regarded as successful, thus enabling a fair comparison among solvers operating at different accuracy levels.

For each problem $p$, a cutoff value is defined as $c_p = F_{\min}^p + \tau (F_{\max}^p - F_{\min}^p)$,
where $F_{\min}^p$ is the best valid objective value across solvers and $\tau > 0$ controls how close to the best value a solver must reach.
Since points satisfying the lower-level accuracy condition  $\text{KKT}_{\mathrm{res}} \le \bar{\zeta}$ typically lie in a neighborhood of convergence, the values of $F_{\min}^p$ and $F_{\max}^p$ may be extremely close, so the latter is computed using the relaxed condition $\text{KKT}_{\mathrm{res}} \le 10^{2}\bar{\zeta}$. 

Both data profiles and performance profiles employ the number of lower-level objective function evaluations as the primary computational metric, as this measure better reflects the computational complexity of solving the problem compared to simply counting upper-level function evaluations. For data profiles, this count is normalized by $n_x \times n_y + 1$.

\subsection{Algorithms implementation details and parameter settings}\label{subsec:algo_impl}

The common parameters are: $\zeta_0 = 0.1$, $\alpha_{\min} = (\sigma \bar{\zeta})^{1/3}$ with $\sigma = 10^{-18}$, $\gamma = 10^{-6}$, $\theta = 0.5$, $\delta = 0.5$.
Our approach to construct the search directions $d_k$ is to employ a dense sequence of orthonormal bases. Specifically, a unit-norm direction is computed by using the Sobol sequence \cite{sobol} and an orthonormal basis is constructed starting from this direction. The algorithm explores all $n$ directions and then generates a new orthonormal basis from a new unit-norm Sobol direction. This should ensure that the resulting directions allow the algorithm to analyze the variable space in a sufficiently uniform and dense way.

For problems involving general upper-level constraints $H(x, y) \leq 0$, we adopt an exact penalty approach where the penalty parameters are managed through a vector $\rho \in \mathbb{R}^m$, following the strategy proposed in \cite{fasano2014linesearch}. The perturbed penalty function is:
\[
R_{\rho}(x, \zeta) = F(x,\tilde y(x,\zeta)) +  \max\{0, \frac{1}{\rho_1}H_1(x, \tilde y(x, \zeta)), \dots, \frac{1}{\rho_m}H_m(x, \tilde y(x, \zeta))\}.
\]
The initial penalty parameters are set according to the constraint violation at the starting point:
\[
(\rho_0)_i = 
\begin{cases}
10^{-3} & \text{if } \max\{0, H_i(x_0, \tilde y (x_0,\zeta_0))\} < 1, \\
10^{-1} & \text{otherwise},
\end{cases}
\quad i = 1, \ldots, m.
\]
During optimization, if the weighted violation  $(\rho_k)_i H_i(x_k, \tilde y(x_k,\zeta_k)) > \max\{\alpha_k, \tilde\alpha_k\}$, then $(\rho_{k+1})_i = 10^{-2} (\rho_k)_i$; otherwise $(\rho_{k+1})_i = (\rho_k)_i$. 

At each upper-level iteration $k$, the lower-level problem 
\[
\min_{y} \ f(x_k, y) \quad \text{subject to} \quad g(x_k, y) \leq 0
\]
is solved using Uno (Unifying Nonlinear Optimization) solver \cite{VanaretLeyffer2024}, where $g(x,y) \leq 0$ encompasses both general inequality constraints and simple bound constraints on $y$. Uno solver is called with a tolerance parameter equal to the current error parameter $\zeta_k$, which is used to control the KKT residual for the lower-level solution. The initial point for solving the lower-level problem at iteration $k$ is set to $y_{k-1}$, the lower-level solution obtained at the previous upper-level iteration.

To take an initial step toward a computational implementation of the proposed methodological framework we compare the described \textbf{DFN-LLA} with several of its possible variants. Each of these variants corresponds to an optimization strategy that differs from that of DFN-LLA. Specifically, we focus on the following three algorithms.
\par\smallskip
\textbf{MS-DFN-LLA} (Multi-Stepsize DFN-LLA), inspired by CS-DFN~\cite{fasano2014linesearch}, exploits the information obtained by sampling the objective function along orthonormal directions. Unlike DFN-LLA, the orthonormal basis is not regenerated immediately after all directions have been explored once. Instead, the algorithm continues to sample the objective function along the same directions using distinct step sizes for each search direction. Directions yielding significant progress retain larger step sizes, whereas directions associated with limited improvement are assigned smaller ones, thereby capturing direction-dependent information about the local geometry of the objective function.
The basis is regenerated only when all step sizes fall below a prescribed threshold $\eta > 0$. This strategy is motivated, for instance, by \cite[Theorem 3.2]{brilli2024complexity}, where the maximum step size can be interpreted as an upper bound on a stationarity measure at the current iterate. Hence, this condition may indicate that sufficient reduction has been achieved along the current directions. When a new basis is generated, the tentative step sizes are initialized to the average of the previous ones and the threshold $\eta$ is reduced by a factor of $0.5$, enabling progressively finer exploration of the search space.
\par\smallskip
\textbf{MS-DFN-NL-LLA} (No-Linesearch variant) removes the extrapolation phase of MS-DFN-LLA. In successful iterations, a step is accepted as soon as it produces sufficient decrease in the objective function, and the new tentative step size is simply divided by $\delta$. The numerical experiments involving MS-DFN-NL-LLA assess whether, once a descent direction has been identified, extrapolating along that direction is more effective than immediately exploring a new one.

\par\smallskip
\textbf{MS-DFN-LLF} (Fixed accuracy variant) keeps the lower-level tolerance fixed at $\zeta_k = \bar{\zeta}$ for the entire optimization procedure.
The numerical experiments carried out with MS-DFN-LLA and MS-DFN-LLF should reveal whether, and under which circumstances, a dynamic control of the accuracy in solving the lower-level problem is advantageous compared to consistently enforcing a high accuracy.

\par\smallskip
The stopping criterion for each variant is satisfied when at least one of the following conditions is met: the number of upper-level function evaluations reaches $1500$, the CPU time exceeds $9000$ seconds, or all step sizes are smaller than $\alpha_{\min}$.

\subsection{Effect of multiple step sizes}\label{subsec:multiple_stepsizes}

To assess the benefit of maintaining multiple step sizes compared to using a single step size for all directions, we consider only problems with $n_x \geq 2$ upper-level variables, since for univariate problems the two configurations are equivalent. Results are reported for a lower-level target precision of $\bar{\zeta} = 10^{-6}$; similar behavior was observed for other precision levels.

\begin{figure}[!htbp]
    \centering
    \setlength{\tabcolsep}{1pt}
    \begin{tabular}{@{}cccc@{}}
        \multicolumn{2}{c}{\small\textit{MS-DFN-LLA vs DFN-LLA}} &
        \multicolumn{2}{c}{\small\textit{MS-DFN-LLA vs MS-DFN-NL-LLA}} \\[2pt]
        \includegraphics[width=0.225\textwidth]{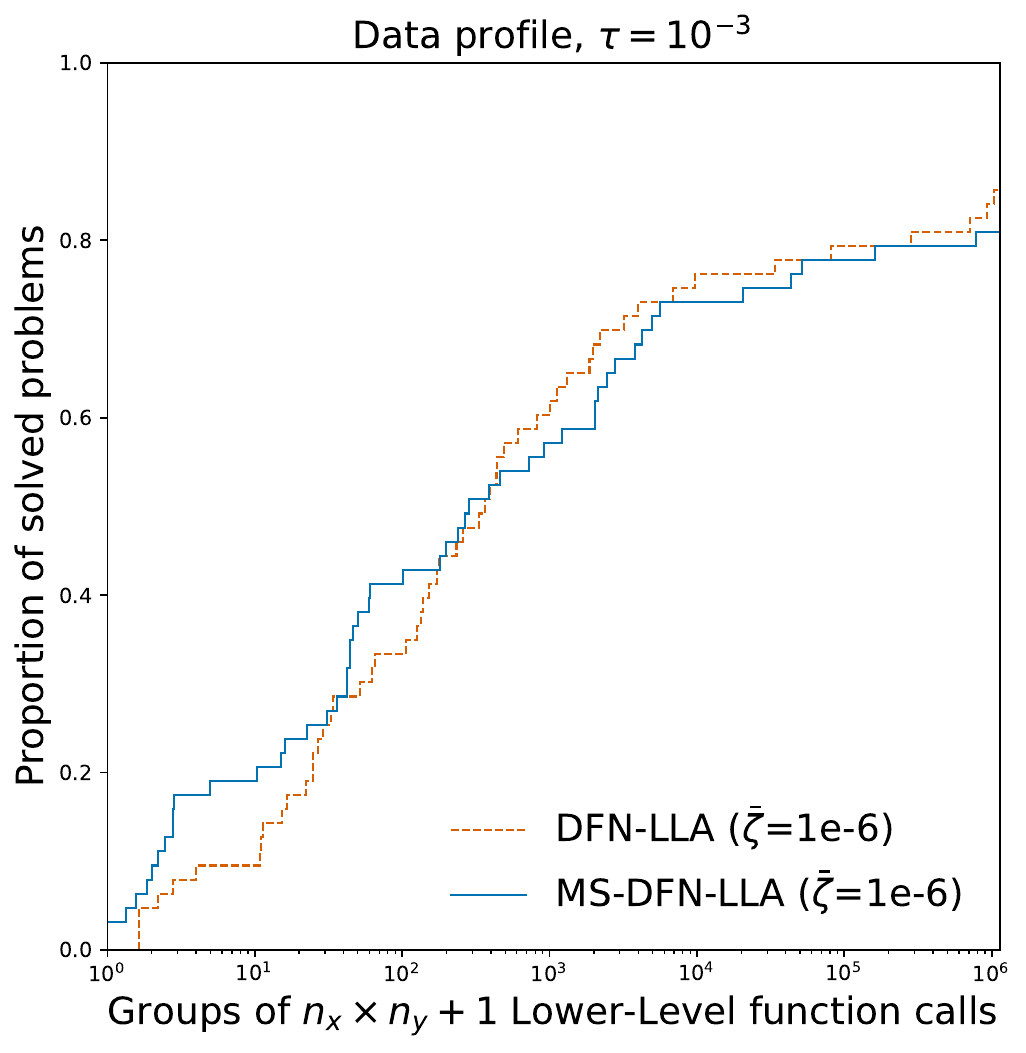} &
        \includegraphics[width=0.225\textwidth]{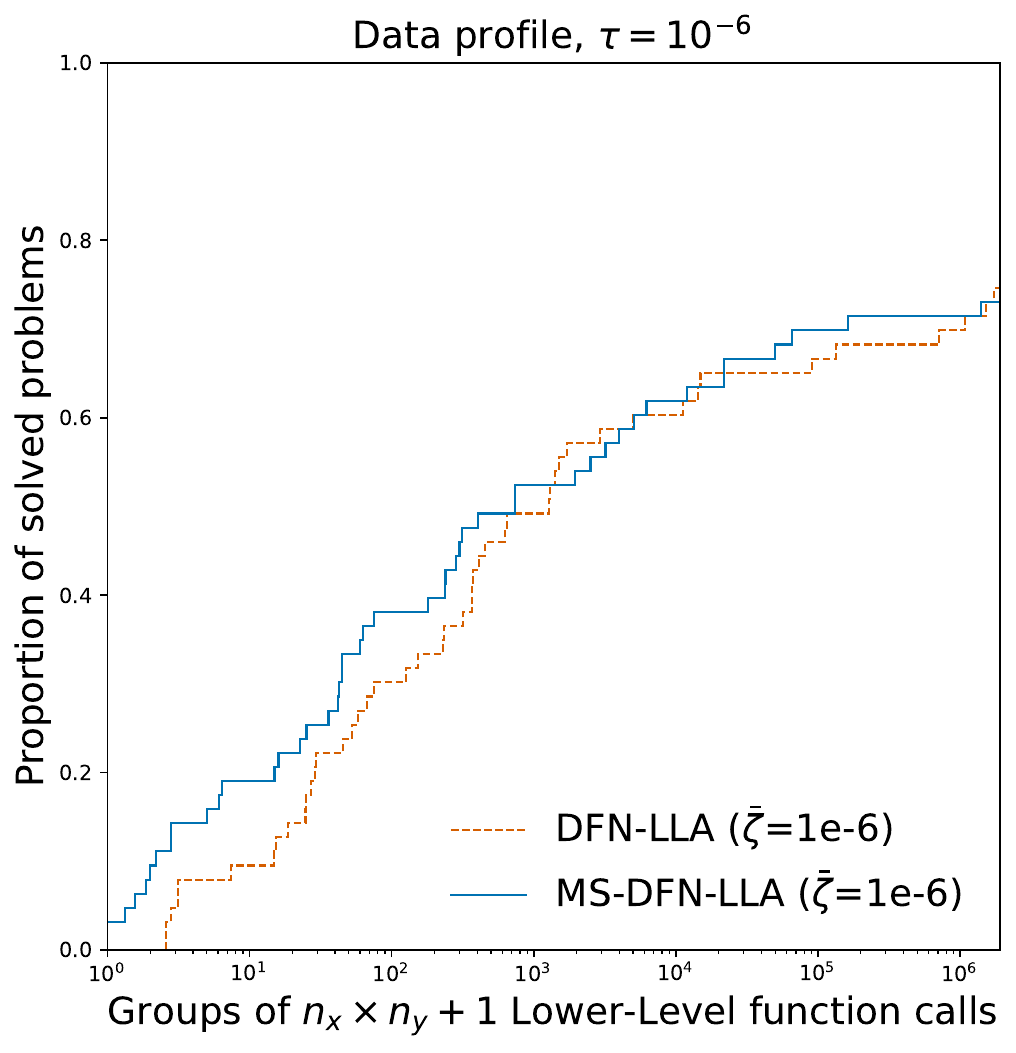} &
        \includegraphics[width=0.225\textwidth]{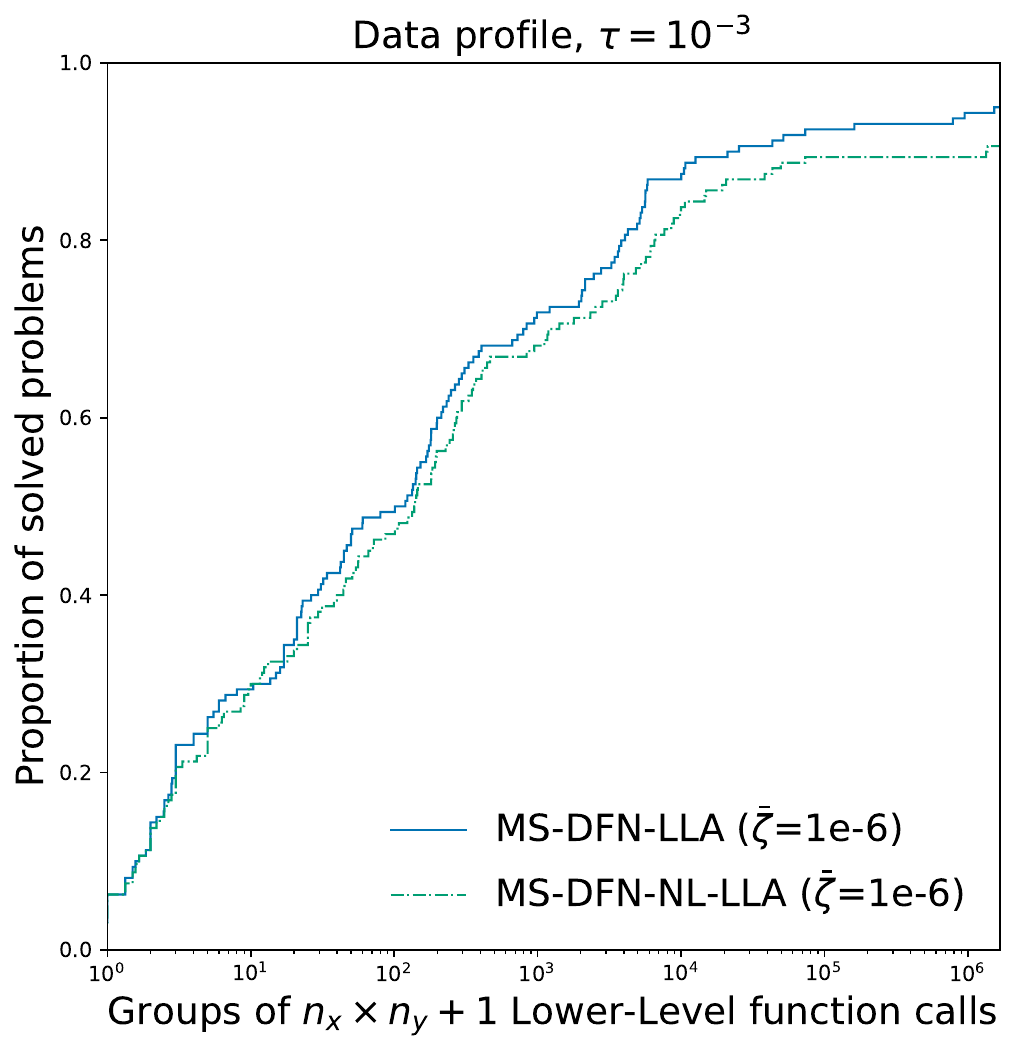} &
        \includegraphics[width=0.225\textwidth]{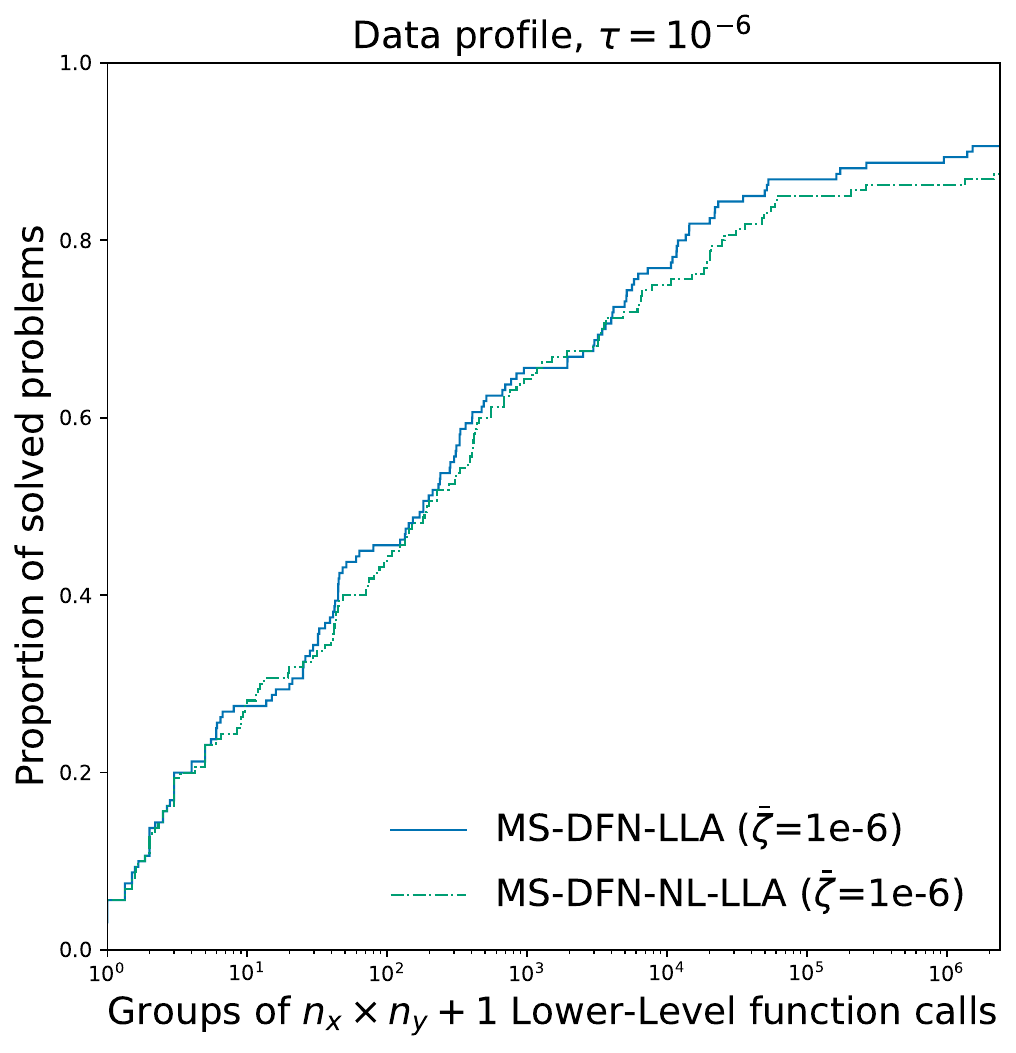} \\[-2pt]
        \includegraphics[width=0.225\textwidth]{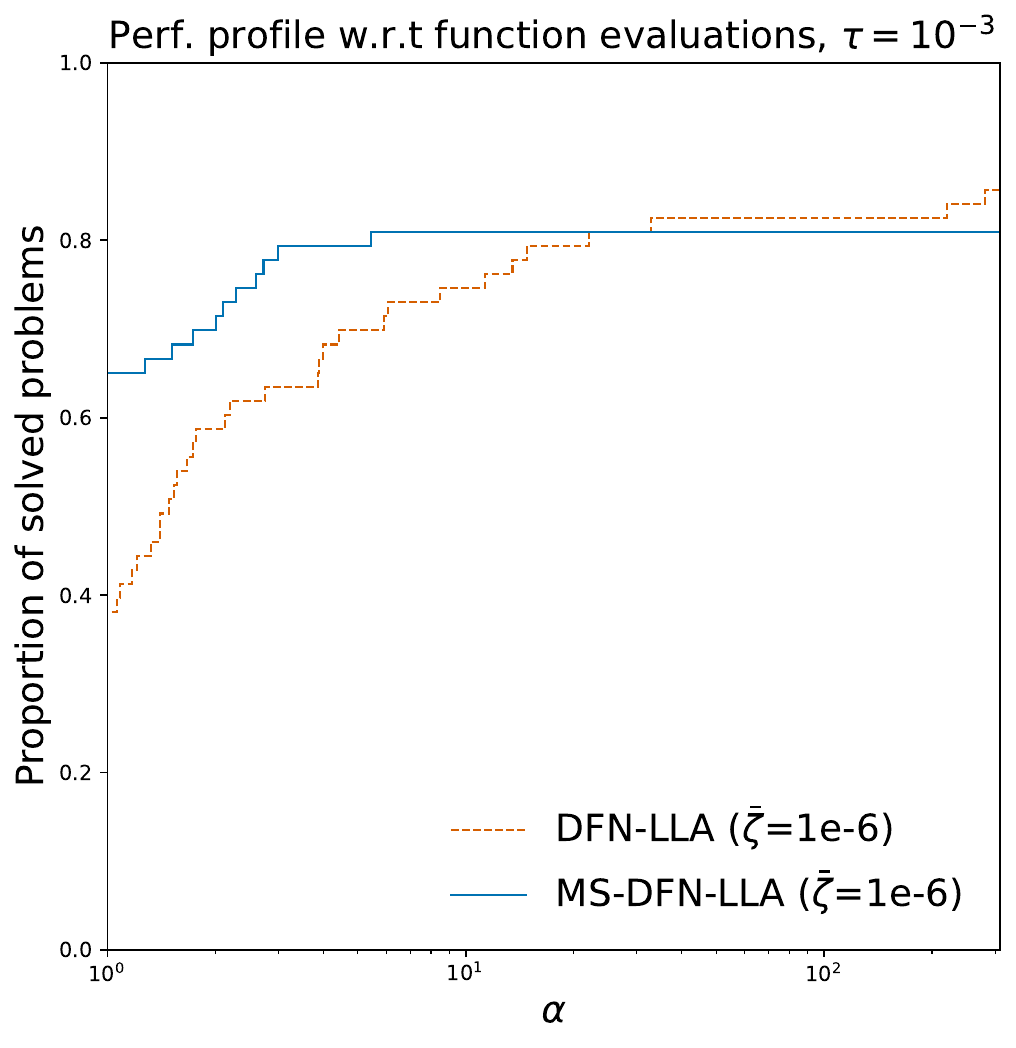} &
        \includegraphics[width=0.225\textwidth]{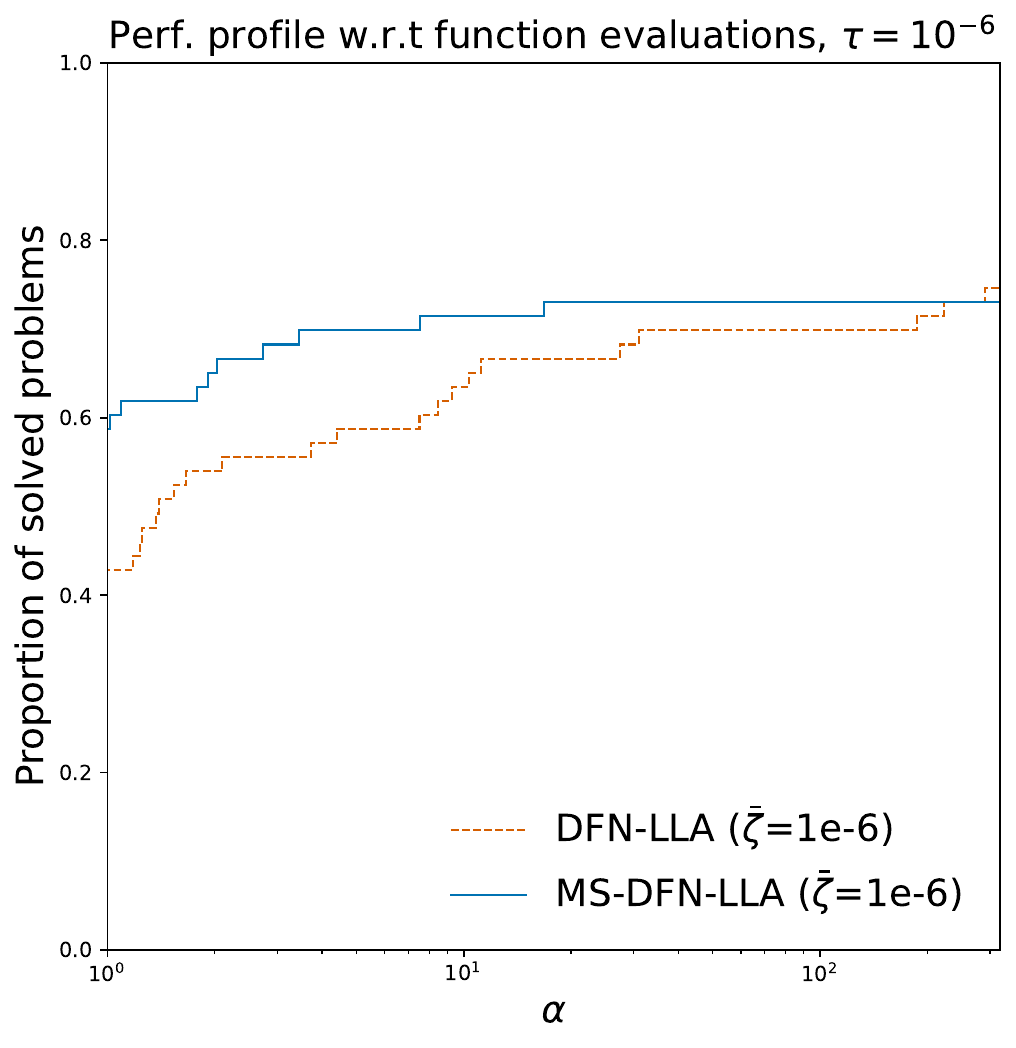} &
        \includegraphics[width=0.225\textwidth]{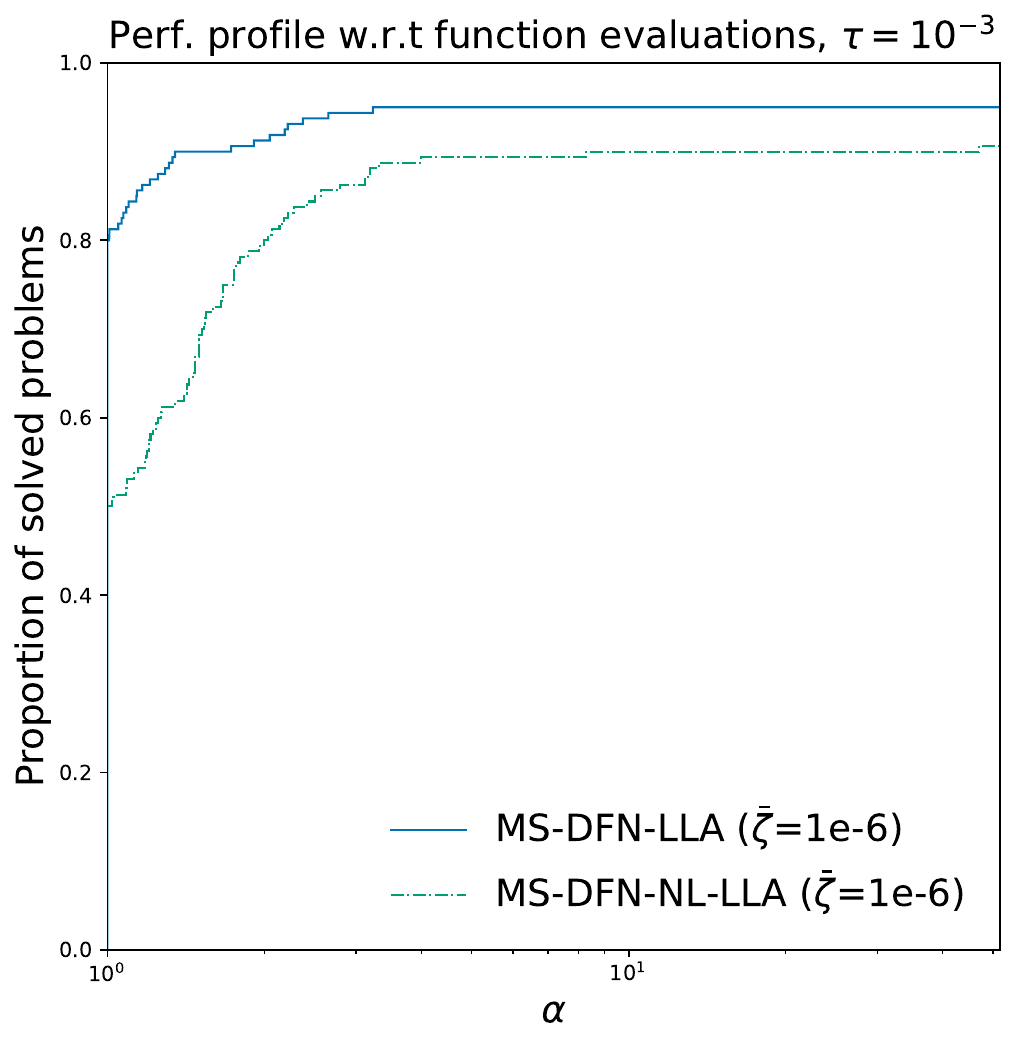} &
        \includegraphics[width=0.225\textwidth]{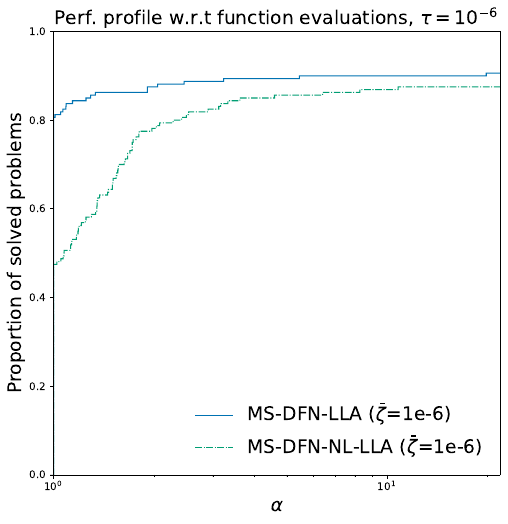}
    \end{tabular}
    \caption{Data profiles (top row) and performance profiles (bottom row) for $\bar\zeta = 10^{-6}$. Left half: MS-DFN-LLA (multiple step sizes) vs.\ DFN-LLA (single step size). Right half: MS-DFN-LLA (with extrapolation) vs.\ MS-DFN-NL-LLA (without). Odd columns: $\tau = 10^{-3}$; even columns: $\tau = 10^{-6}$.}
    \label{fig:results_piualfa_linesearch}
\end{figure}

As shown in Figure~\ref{fig:results_piualfa_linesearch} (left half), maintaining a memory of the step sizes for each direction and regenerating the orthonormal basis only after sufficient exploration improves performance. The multiple step size strategy better exploits direction-dependent information about the objective, yielding more efficient exploration. Consequently, all subsequent comparisons use the multiple step size variant (MS-DFN-*).

\subsection{Effect of Projected Extrapolation}\label{subsec:linesearch}

The comparison between MS-DFN-LLA and its variant without linesearch extrapolation (MS-DFN-NL-LLA) reveals the contribution of the extrapolation procedure. As before, results are reported for a lower-level target precision of $\bar{\zeta} = 10^{-6}$, as similar trends were observed across all precision levels.

Figure~\ref{fig:results_piualfa_linesearch} (right half) shows that the extrapolation phase is beneficial: by taking larger steps when descent is detected, it keeps step sizes larger for more iterations, delaying the reduction of $\zeta$ and allowing more progress with coarser (cheaper) lower-level solutions.

\subsection{Comparison of fixed vs.\ adaptive accuracy}

This subsection addresses the central question of the computational study: quantifying the benefit of dynamic accuracy adaptation compared to always solving the lower-level problem to the target precision. To this end, we compare the two strategies described in Section~\ref{subsec:algo_impl}, namely MS-DFN-LLF and MS-DFN-LLA, across three KKT levels: $10^{-3}$, $10^{-6}$, $10^{-9}$. Figure~\ref{fig:results_adaptive} collects all data and performance profiles for $\tau \in \{10^{-3}, 10^{-6}\}$.

Already at the coarsest tolerance ($\bar\zeta = 10^{-3}$, top row), the adaptive strategy shows an advantage that grows at intermediate ($\bar\zeta = 10^{-6}$, middle row) and high precision ($\bar\zeta = 10^{-9}$, bottom row), where it avoids solving the lower-level to extreme accuracy during early iterations. 

\begin{figure}[!htbp]
    \centering
    \setlength{\tabcolsep}{1pt}
    \begin{tabular}{@{}cccc@{}}
        \includegraphics[width=0.225\textwidth]{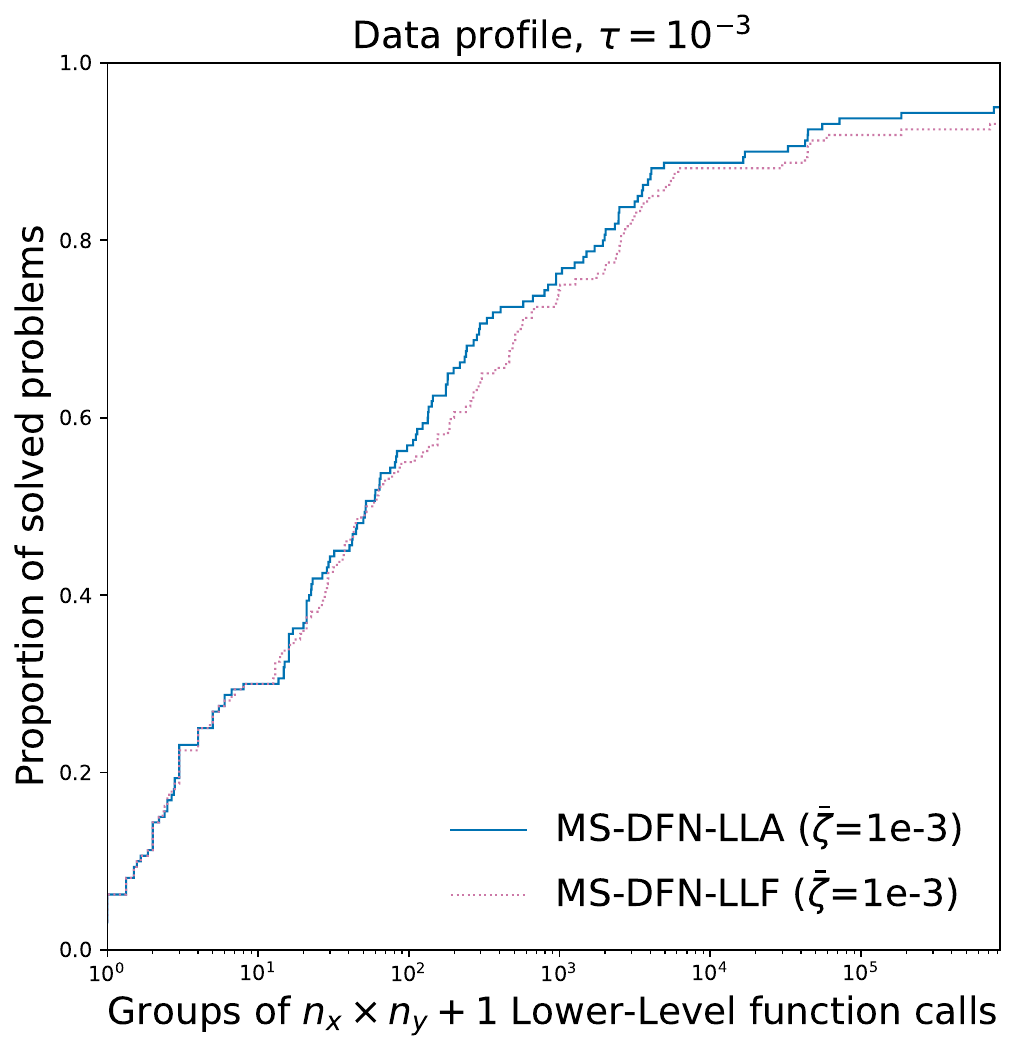} &
        \includegraphics[width=0.225\textwidth]{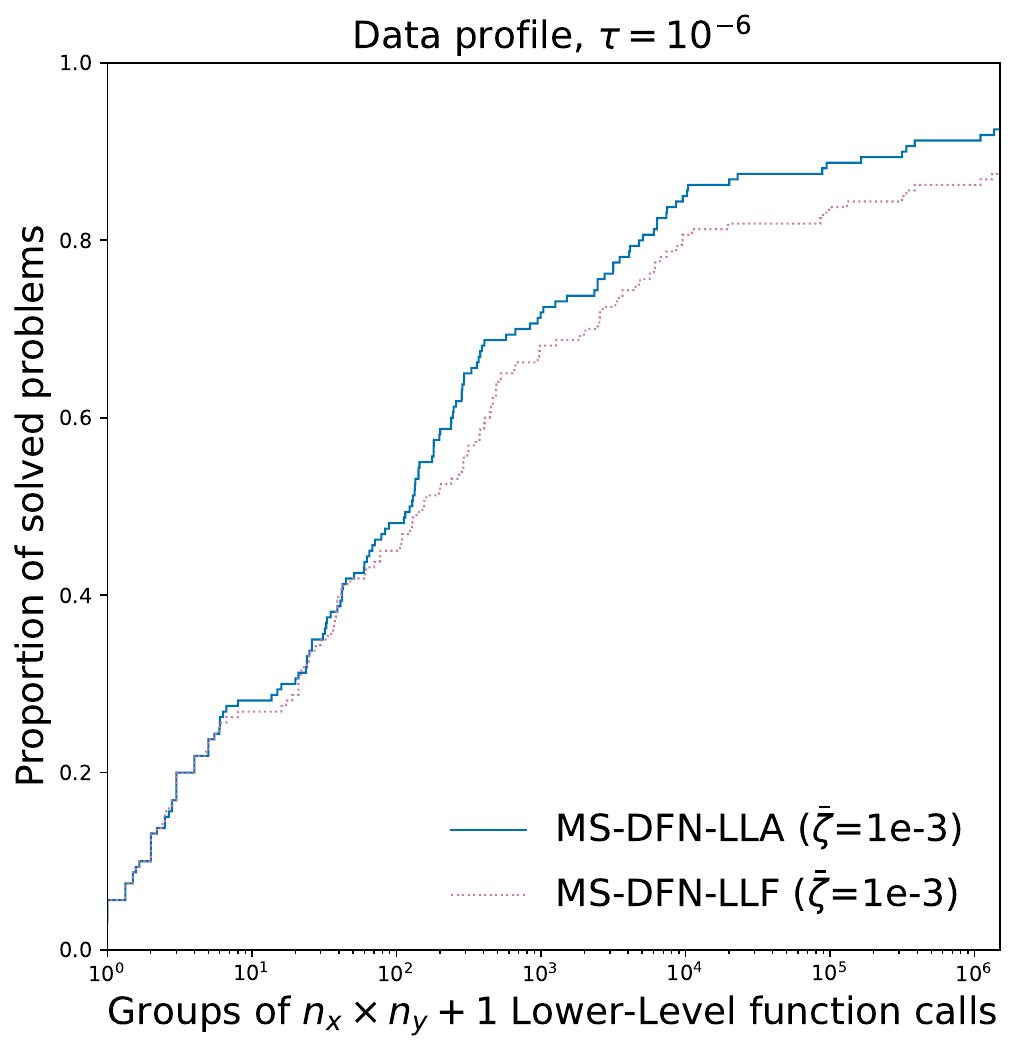} &
        \includegraphics[width=0.225\textwidth]{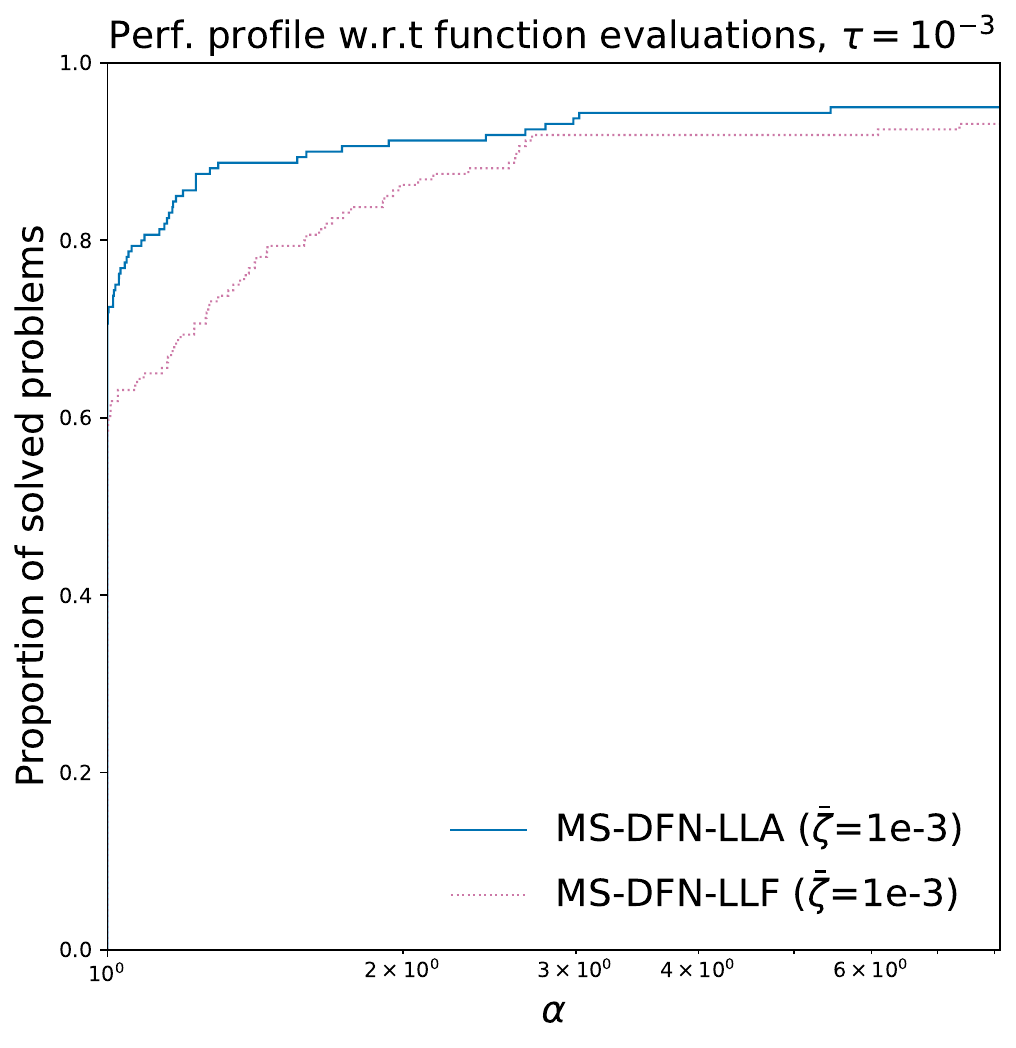} &
        \includegraphics[width=0.225\textwidth]{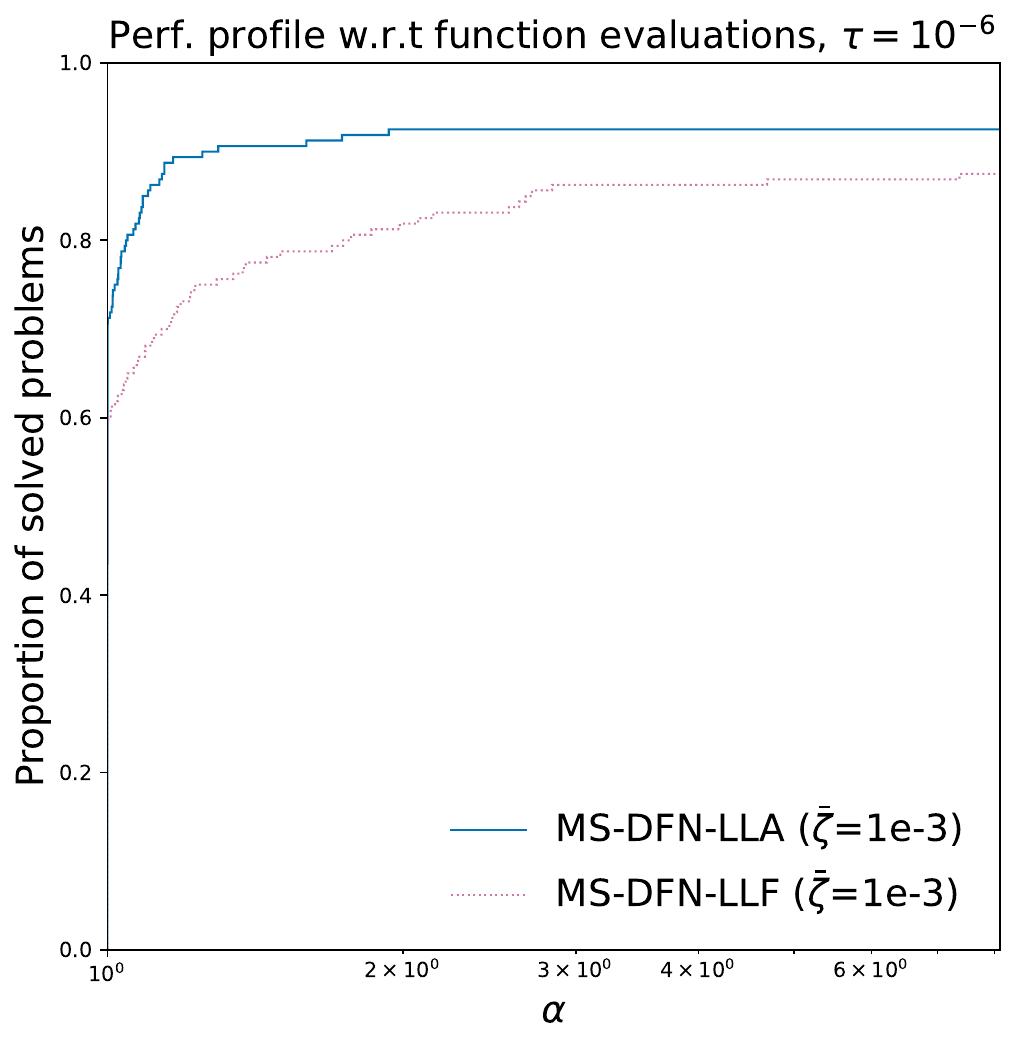} \\[-2pt]
        \includegraphics[width=0.225\textwidth]{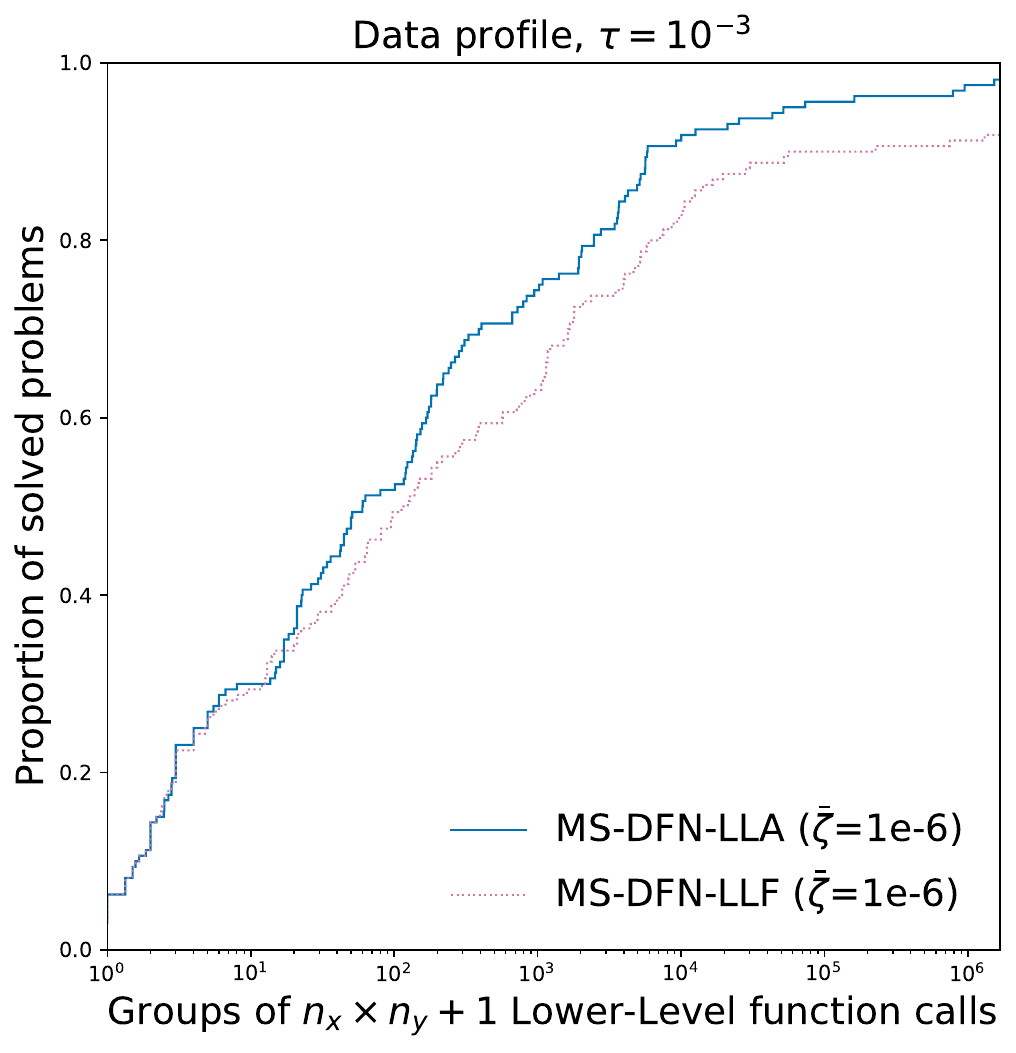} &
        \includegraphics[width=0.225\textwidth]{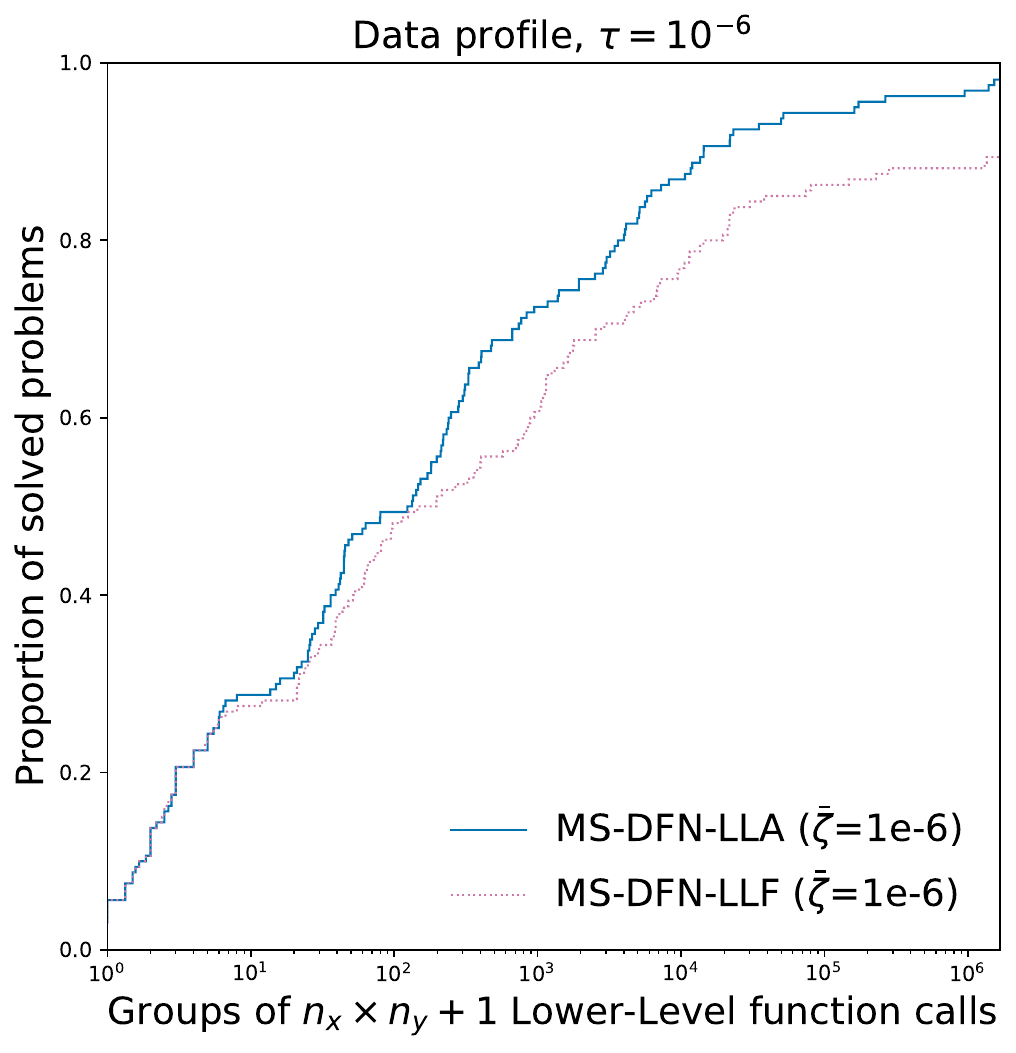} &
        \includegraphics[width=0.225\textwidth]{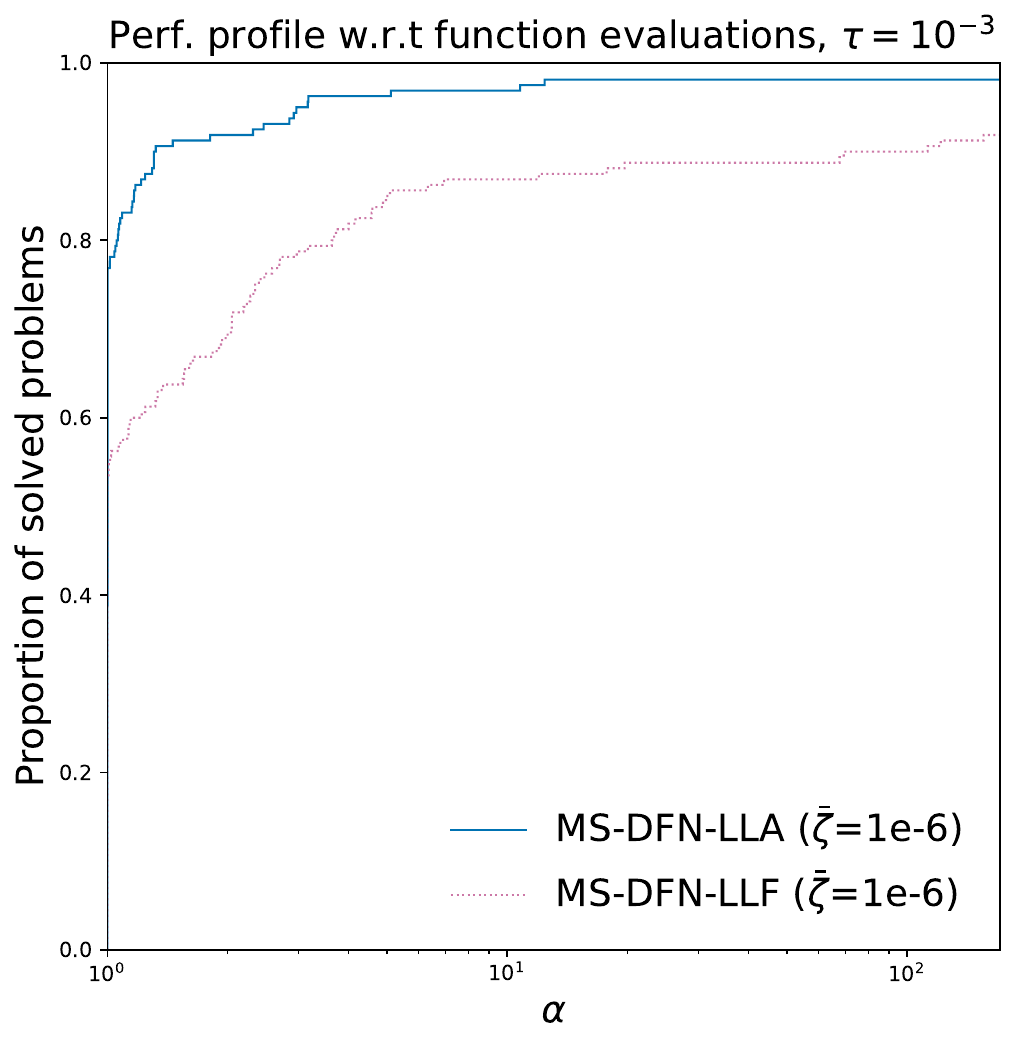} &
        \includegraphics[width=0.225\textwidth]{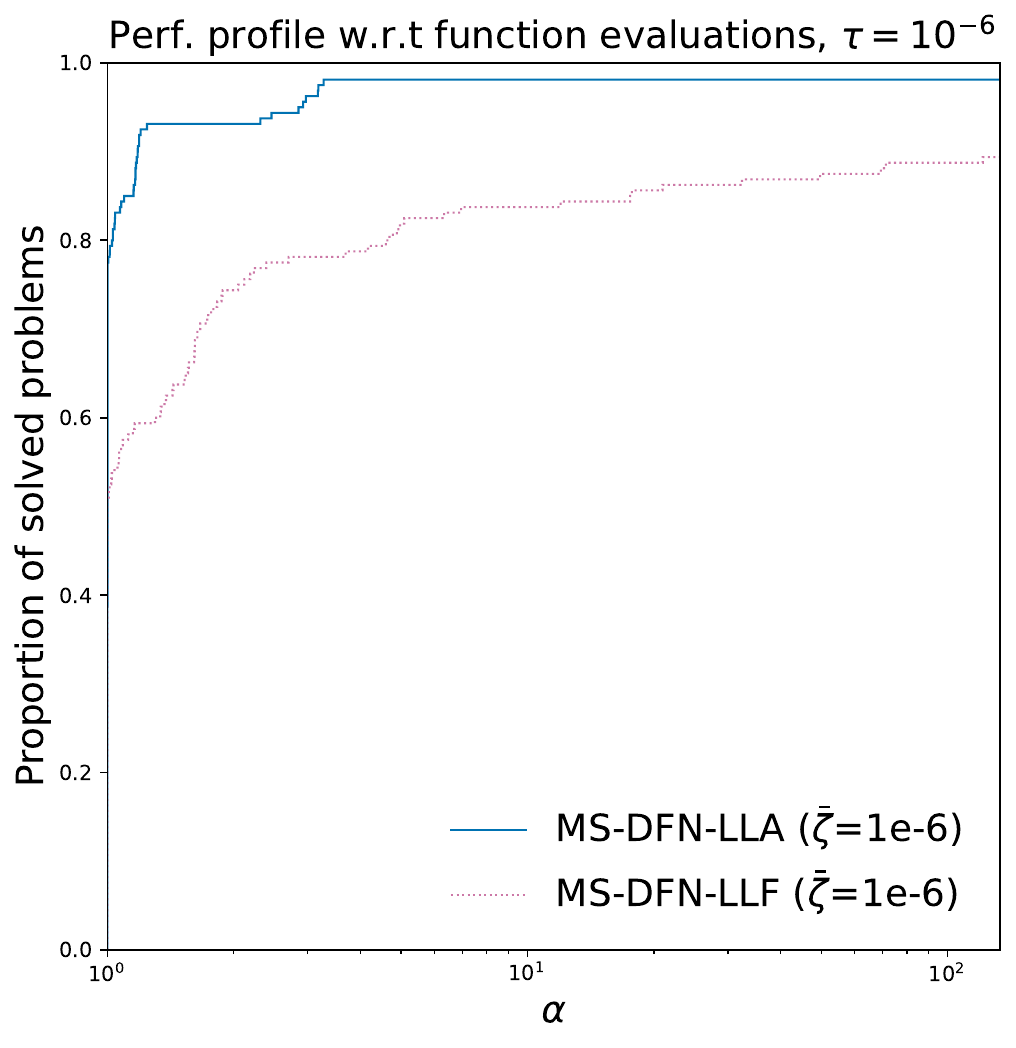} \\[-2pt]
        \includegraphics[width=0.225\textwidth]{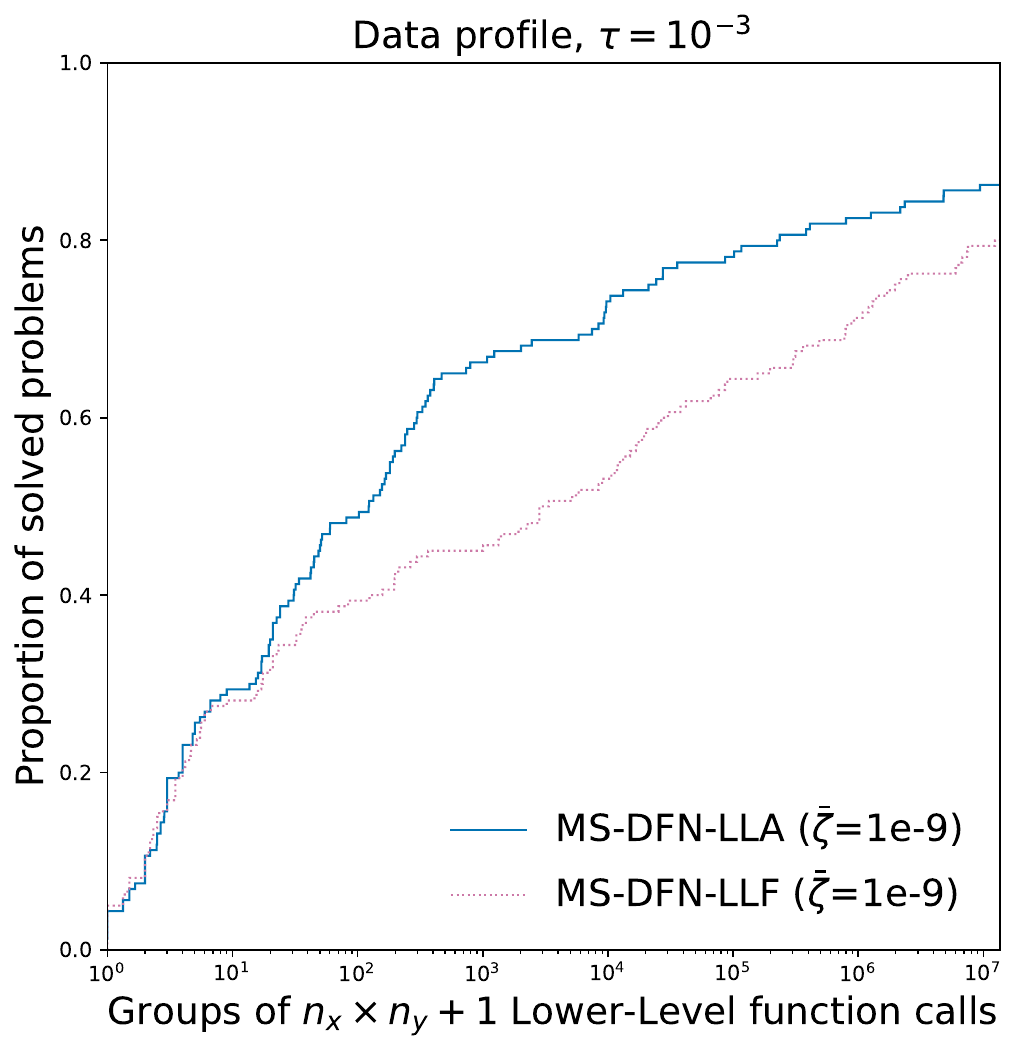} &
        \includegraphics[width=0.225\textwidth]{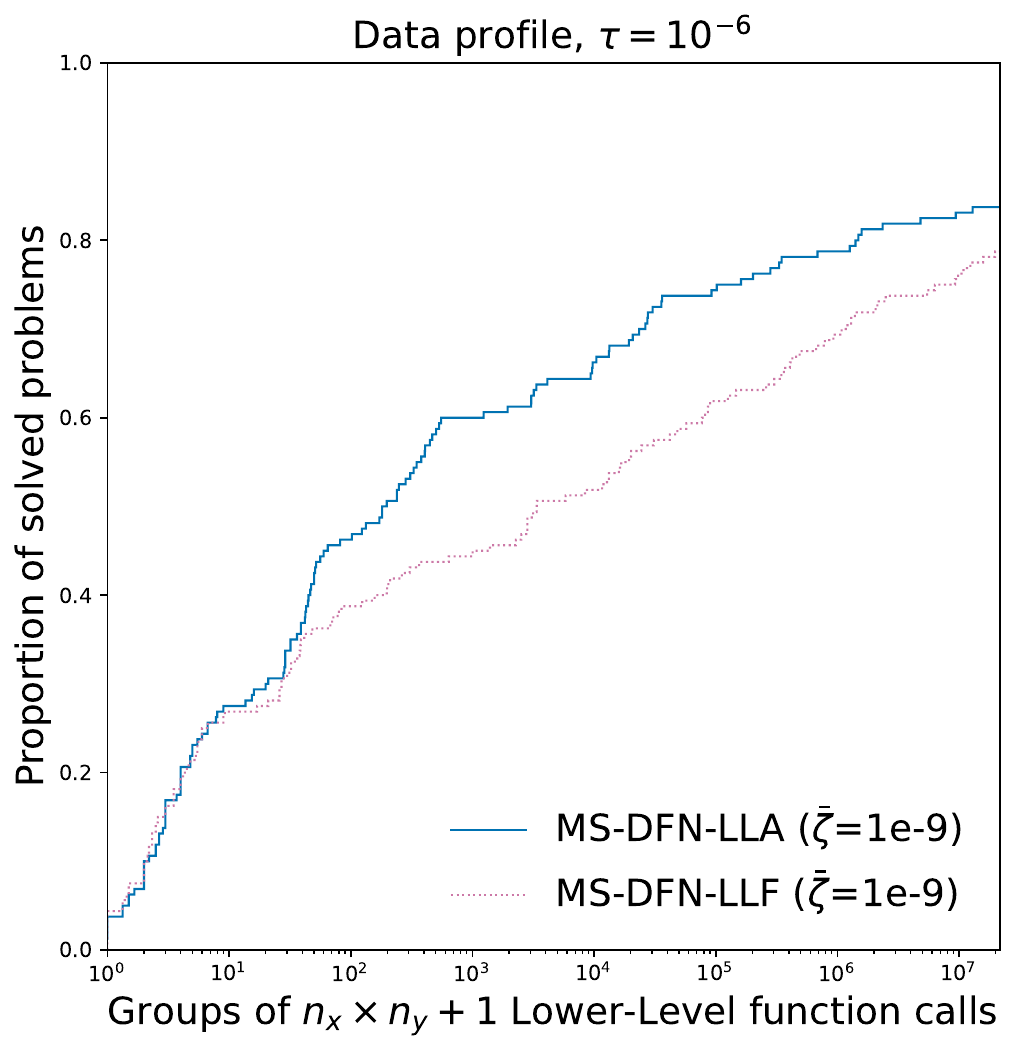} &
        \includegraphics[width=0.225\textwidth]{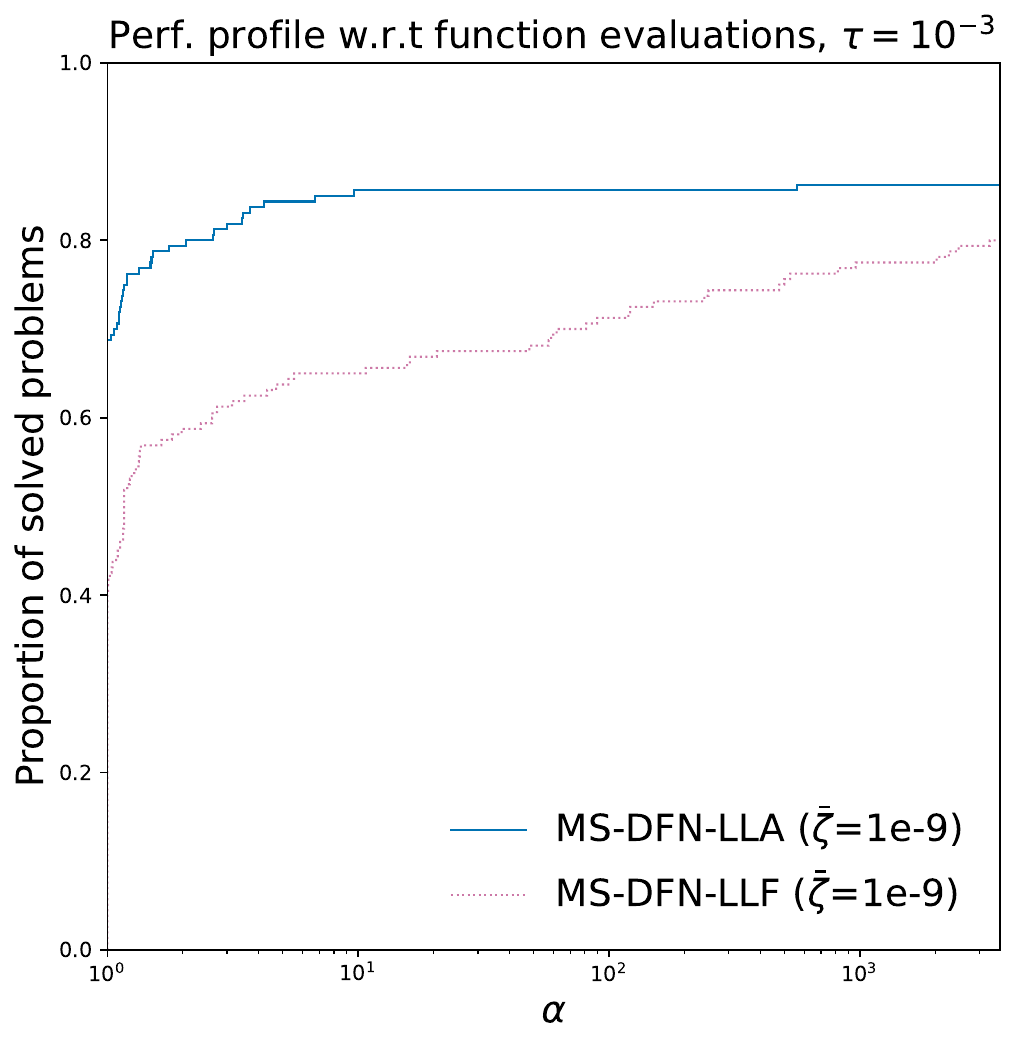} &
        \includegraphics[width=0.225\textwidth]{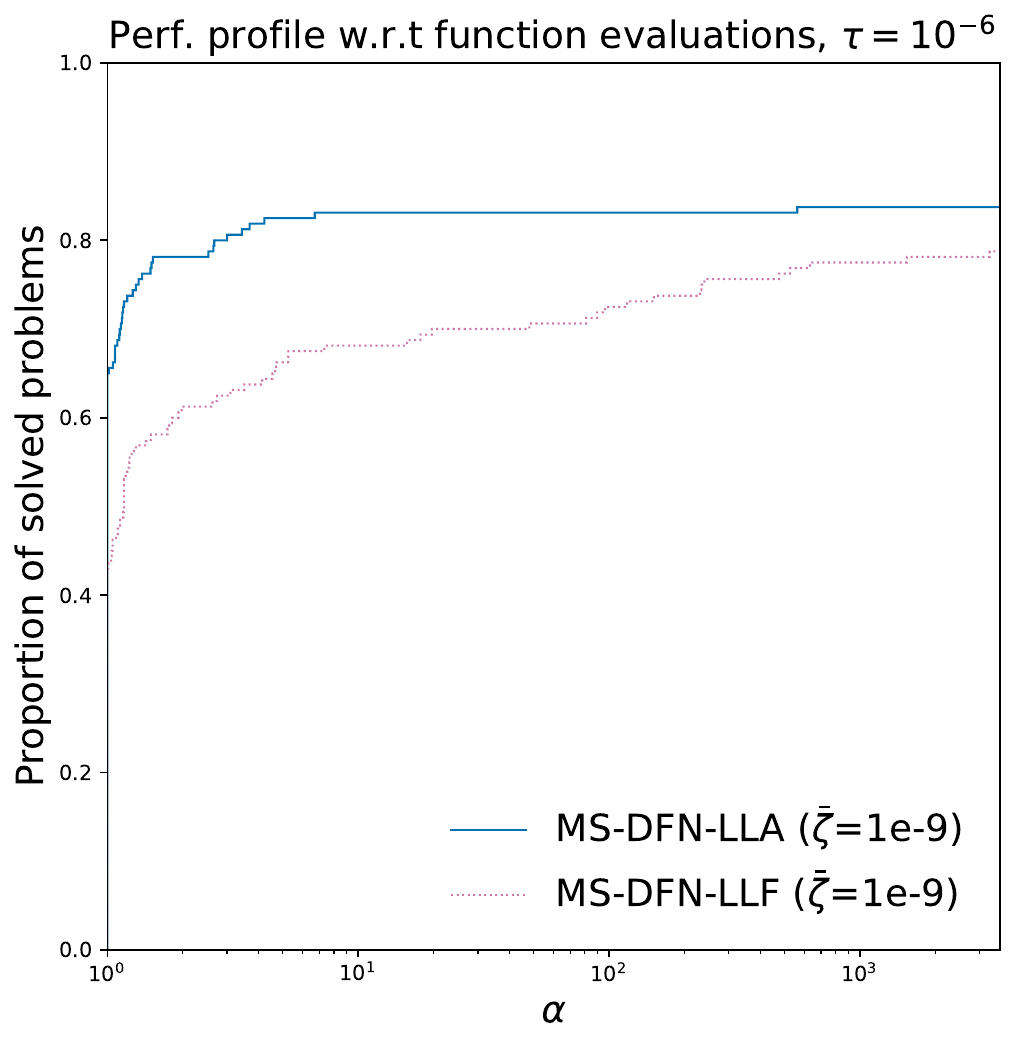}
    \end{tabular}
    \caption{Data profiles (columns 1--2) and performance profiles (columns 3--4) comparing MS-DFN-LLF (fixed accuracy) with MS-DFN-LLA (adaptive accuracy). Rows correspond to $\bar\zeta = 10^{-3}$ (top), $10^{-6}$ (middle), $10^{-9}$ (bottom). Odd columns: $\tau = 10^{-3}$; even columns: $\tau = 10^{-6}$.}
    \label{fig:results_adaptive}
\end{figure}

To further illustrate the practical impact of adaptive accuracy, the \textit{Optimal Control} problem from the BOLIB collection~\cite{Zhou2020} is considered, configured with $n_y = 72$ lower-level variables so that each lower-level solve is computationally expensive. Table~\ref{tab:optcontrol} reports the performance of MS-DFN-LLA and MS-DFN-LLF for target tolerances.

\begin{table}[!htbp]
\centering
\small
\setlength{\tabcolsep}{4pt}
\begin{tabular}{l@{\;\;}cc@{\;\;}cc@{\;\;}cc}
\toprule
& \multicolumn{2}{c}{$\bar{\zeta} = 10^{-3}$}
& \multicolumn{2}{c}{$\bar{\zeta} = 10^{-6}$}
& \multicolumn{2}{c}{$\bar{\zeta} = 10^{-9}$} \\
\cmidrule(lr){2-3} \cmidrule(lr){4-5} \cmidrule(lr){6-7}
& LLA & LLF
& LLA & LLF
& LLA & LLF \\
\midrule
$n_f$ at best $F$
& 618   & 1\,019
& 3\,065 & 6\,289
& 1\,929 & -- \\

CPU time (s)
& 598   & 1\,033
& 3\,232 & 5\,854
& 2\,045 & -- \\

Best $F$
& 0.52852426 & 0.52799624
& 0.52948840 & 0.52948872
& 0.52949427 & -- \\


KKT residual
& $4.20 \times 10^{-4}$ & $9.08 \times 10^{-4}$
& $9.99 \times 10^{-7}$ & $9.44 \times 10^{-7}$
& $6.87 \times 10^{-8}$ & -- \\
\bottomrule
\end{tabular}
\caption{Performance of MS-DFN-LLA and MS-DFN-LLF on the Optimal Control problem ($n_y = 72$) for three target tolerances. $n_f$: number of lower-level function evaluations at the best upper-level objective found. A dash (--) indicates that no valid solution was returned.}
\label{tab:optcontrol}
\end{table}

For $\bar{\zeta} = 10^{-3}$, the adaptive strategy reaches its best objective in $618$ lower-level evaluations versus $1\,019$ for the fixed approach, reducing CPU time from $1\,033$\,s to $598$\,s; the adaptive method also achieves a smaller KKT residual.
For $\bar{\zeta} = 10^{-6}$, the gap widens: convergence requires $3\,065$ versus $6\,289$ evaluations, with nearly identical objective values (matching to six digits).
For $\bar{\zeta} = 10^{-9}$, the fixed strategy fails to produce a solution, while the adaptive approach converges to a point with KKT residual of order $10^{-8}$.

The experimental results demonstrate that the benefit of dynamic error adaptation is already evident at coarse precision levels and increases as the required precision becomes more stringent.

\section{Conclusions}\label{sec:conclusions}

We proposed a derivative-free algorithmic framework for bilevel optimization in which upper-level objective and constraint functions are black boxes. It relies on a single-level reformulation via the optimal lower-level reaction, combined with a linesearch-based direct-search method. A key feature is the adaptive accuracy strategy: rather than solving the lower-level to a fixed tolerance, the algorithm couples the lower-level tolerance~$\zeta_k$ with the step-size parameters, starting with a relaxed accuracy and tightening it as convergence progresses.
We proved that when $\zeta_k \to 0$ the algorithm's limit points are Clarke-Jahn stationary, and when $\zeta_k \geq \bar\zeta > 0$ convergence to approximate stationary points holds in a Goldstein sense. These results were extended to nonlinear upper-level constraints via exact penalty, yielding convergence to Clarke-KKT points under suitable constraint qualifications.
The computational study on 160~BOLIB problems confirmed that the adaptive strategy consistently outperforms the fixed-accuracy variant across all tolerance levels, with the advantage growing as precision requirements increase. At $\bar\zeta = 10^{-9}$ the fixed strategy often fails, while the adaptive one remains viable, showing that dynamic accuracy management may be essential at high precision.
Future work includes analysis under smoothness assumptions on the upper-level objective and extension to MADS-type schemes.

\section*{Acknowledgments}

E. Cesaroni acknowledges financial support from the European Union – Next Generation EU, Mission 4 Component 1 CUP B53C23001740006.

\bibliographystyle{siamplain}
\bibliography{bilevel}

\end{document}